\documentclass{amsart}
\usepackage{amssymb}
\usepackage{eucal}
\usepackage{amsfonts}
\usepackage{epsfig}
\usepackage[frame,ps,dvips,matrix,arrow,curve,rotate]{xy}
\let\oldcite\cite                                  

\newtheorem{thm}{Theorem}[section]
\newtheorem{cor}[thm]{Corollary}
\newtheorem{lem}[thm]{Lemma}

\newtheorem{prop}[thm]{Proposition}
\theoremstyle{definition}
\newtheorem{defn}[thm]{Definition}
\theoremstyle{remark}
\newtheorem{rem}[thm]{Remark}
\numberwithin{equation}{section}
\newtheorem{ex}[thm]{Example}
\theoremstyle{remark}

\newcommand{\liminv}{\varprojlim}
\let\:=\colon
\newcommand{\x}{\times}

\newcommand{\Aut}{\operatorname{Aut}}

\newcommand{\PSL}{\operatorname{PSL}}
\newcommand{\id}{\operatorname{id}}
\newcommand{\pr}{\operatorname{pr}}
\newcommand{\Iso}{\operatorname{Iso}}

\def\smashedlongrightarrow{\setbox0=\hbox{$\longrightarrow$}\ht0=1pt\box0}
\def\risom{\buildrel\sim\over{\smashedlongrightarrow}}

\def\smashedst{\setbox0=\hbox{$\rightrightarrows$}\ht0=4pt\box0}
\newcommand{\sst}[1]{\stackrel{#1}{\smashedst}}

\newcommand{\lra}{\longrightarrow}
\newcommand{\llra}[1]{\stackrel{#1}{\lra}}

\newcommand{\bbZ}{\mathbb{Z}}
\newcommand{\bbR}{\mathbb{R}}
\newcommand{\bbC}{\mathbb{C}}
\newcommand{\bbN}{\mathbb{N}}

\newcommand{\bbQ}{\mathbb{Q}}

\newcommand{\al}{\alpha}

\newcommand{\ep}{\epsilon}
\newcommand{\De}{\Delta}

\newcommand{\X}{\mathcal{X}}
\newcommand{\Y}{\mathcal{Y}}
\newcommand{\Z}{\mathcal{Z}}
\newcommand{\U}{\mathcal{U}}
\newcommand{\V}{\mathcal{V}}
\newcommand{\C}{\mathsf{C}}
\newcommand{\OO}{\mathcal{O}}
\newcommand{\HH}{\mathcal{H}}
\newcommand{\N}{\mathcal{N}}

\newcommand{\Xm}{\X_{mod}}
\newcommand{\Ym}{\Y_{mod}}
\newcommand{\pim}{\pi}

\newcommand{\pX}{\pi_1(\X,x)}
\newcommand{\pXm}{\pi_1(\Xm,x)}
\newcommand{\pY}{\pi_1(\Y,y)}

\newcommand{\ox}{\omega_x}
\newcommand{\oy}{\omega_y}

\newcommand{\GSet}{G\mbox{-}\mathsf{Set}}
\newcommand{\GWSet}{G^{\wedge}\mbox{-}\mathsf{Set}}

\begin{document}

\title[Fundamental groups of topological stacks with slice property]
{Fundamental groups of topological stacks with slice property}%
\author{Behrang Noohi}

\begin{abstract}
  The main result of the paper is a formula for  the fundamental
  group of the coarse moduli space of a topological stack. As an
  application, we find simple formulas for the fundamental group of
  the coarse quotient  of a group action on a topological space in
  terms of the fixed point data. In particular, we recover, and
  vastly generalize, results of
  \oldcite{Armstrong1,Armstrong2,Bass,Higgins,Rhodes}.
\end{abstract}
\maketitle%

\tableofcontents%
\section{Introduction}{\label{S:Introduction}}

The purpose of this paper is to prove a basic formula for the
fundamental group of the coarse moduli stack of a topological stack
(Theorem \ref{T:semilocally}). This result has  consequences in
classical algebraic topology which seem, surprisingly, to be new.
(Some special cases have appeared previously in
\cite{Armstrong1,Armstrong2,Higgins,Rhodes}.) They give rise to
simple formulas for  the fundamental group of the coarse quotient
space of a group action on a topological space in terms of the fixed
point data of the action; see Theorem \ref{T:fundgpquotient1},
Theorem \ref{T:fixed2}, and Remark \ref{R:trick} -- Corollary
\ref{C:connected} should also be of interest.

For the above results to hold, one needs a technical hypothesis which
is called the {\em slice condition} (Definitions \ref{D:mild} and
\ref{D:strong}). As the terminology suggests, this notion is modeled
on the slice property of compact Lie group actions. Stacks which
satisfy slice condition include  quotient stacks of proper Lie
groupoids and stack that are (locally) quotients of Cartan
$G$-spaces ($\S$\ref{SS:slice}, Example 2), where $G$ is  an
arbitrary Lie group. The latter case includes: 1) Deligne-Mumford
topological stacks, hence, all orbifolds; 2) quotient stacks of
compact Lie group actions on completely regular spaces; 3) quotient
stacks of proper Lie group actions on locally compact spaces.

Roughly speaking, the idea behind the above theorems is that, in the
presence of the slice condition, the fundamental group of the coarse
moduli space (or the coarse quotient of a group action) is obtained
by simply killing the loops which ``trivially'' die as we pass to the
coarse space. For example, we know that under the moduli map
$\X\to\Xm$ the {\em ghost loops} (or {\em inertial loops}) in
$\pi_1(\X)$  die. Theorem \ref{T:semilocally} says, roughly,
that $\pi_1(\Xm)$ is obtained precisely  by killing all the ghost
loops in $\pi_1(\X)$.

Theorem \ref{T:semilocally} can be applied in a variety of
situations, e.g., when $\X$ is a complex-of-groups, an orbifold, the
leaf stack of a foliation, and so on. In the case where $\X$ is a
graph-of-groups, this recovers a result of Bass (\cite{Bass},
Example 2.14).

Our strategy in proving these results is to make a systematic use of
the covering theory of topological stacks, as developed in
\cite{homotopy}, $\S$18. This approach  has the advantage that it is
neat and it minimizes the use of path chasing arguments. The major
players in the game are the maps $\ox \: I_x \to \pX$ introduced in
(\cite{homotopy}, $\S$17) which realize elements of the inertia
groups $I_x$ as ghost loops in $\X$.

\vspace{0.1in}

\noindent{\bf Organization of the paper} To be able to make use of
the formalism of Galois categories, in the first part of
$\S$\ref{S:AbstractGal} we go over prodiscrete topologies on groups
and prodiscrete completions. (A prodiscrete topology on a groups is
a topology which admits a basis at the identity consisting  of subgroups
(not necessarily normal); pro-$\mathsf{C}$-topologies of
\cite{Ribes} are examples of these.) This is presumably standard
material. In the second part of $\S$\ref{S:AbstractGal}, we remind
the reader of Grothendieck's theory of Galois categories.

Main examples of Galois categories arise from covering stacks, and
they give rise to prodiscrete groups. This is discussed in
$\S$\ref{S:Galois}. In $\S$\ref{S:GaloisExamples} we introduce the
Galois categories that interest us in this paper. Understanding
these Galois categories is the key in proving our results about
fundamental groups of stacks.

Up to this point in the paper, everything is quite formal and we do
not make any assumptions on our stacks (other than being connected
and locally path-connected). In $\S$\ref{S:Mild}, we introduce
topological stacks with  {\em slice property} (Definitions
\ref{D:mild} and \ref{D:strong}). In $\S$\ref{S:Strongly}, we look
at a certain class of stacks with slice property which satisfy a
certain locally path-connectedness condition. We call them {\em
strongly locally path-connected} stacks. These are stacks whose
covering theory is as well-behaved
as it can  be.

The  interrelationship between the various Galois categories
introduced in $\S$\ref{S:GaloisExamples} is discussed in
$\S$\ref{S:Relation}. The slice property will play an important
role in relating these Galois theories.

In $\S$\ref{S:Moduli}, we translate the results of
$\S$\ref{S:GaloisExamples} in terms of fundamental groups and
obtain our first main result, Theorem \ref{T:semilocally}.

In $\S$\ref{S:Quotient1} and $\S$\ref{S:Quotient2}, we apply
Theorem \ref{T:semilocally} to the quotient stack $[X/G]$ of a
group action and derive our next main results,
Theorem \ref{T:fundgpquotient1} and Theorem \ref{T:fixed2}.

In $\S$\ref{S:Armstrong}, we compare our results with those of
Armstrong \cite{Armstrong2}.

\section{Some abstract Galois theory}{\label{S:AbstractGal}}

In this section, we review Grothendieck's   theory of Galois
categories, slightly modified so we do not need the finiteness
assumptions of \cite{SGA1}.

Throughout the paper,
all group actions are on the left and are continuous.

\subsection{Prodiscrete completions}{\label{SS:pro}}

To set up a Galois theory that is general enough for our purposes,
we need to extend the theory of profinite groups so that it applies
to topologies generated by families of subgroups that are not
necessarily of finite index. Lack of compactness causes a bit of
technical difficulty, but it does not effect the outcome, at least
as far as our applications are concerned. The material in
 this subsection should be standard.

 \begin{defn}{\label{D:prodiscrete}}
   Let $G$ be a topological group. We say that the topology of $G$
   is {\bf prodiscrete} if its open subgroups form a fundamental
   system of neighborhoods  at the identity.
 \end{defn}

Let $G$ be a group, and let $\HH$ be a family of subgroups of $G$
satisfying the following axioms:

  \begin{itemize}

          \item[$\mathbf{Top1.}$] If $H_1, H_2 \in \HH$,
                then $H_1 \cap H_2 \in \HH$.

          \item[$\mathbf{Top2.}$] For any $H \in \HH$ and any $g \in
           G$,     $gHg^{-1}$ is in $\HH$.

           \item[$\mathbf{Top3.}$] If $H \in \HH$ and $H \subseteq
            H'$,  then $H' \in \HH$.
   \end{itemize}

In this case, there is a prodiscrete topology on $G$ with the property that
a subgroup $H \subseteq G$ is open if and only if $H \in \HH$. We
sometime refer to this topology as the {\bf $\HH$-topology} on $G$.

The $\HH$-topology is Hausdorff if and only if
$\bigcap_{\HH}H=\{1\}$. In this case, the topology is totally
disconnected. It is compact if and only if every $H \in \HH$ has
finite index in $G$. In this case, there exists a fundamental system
of neighborhoods at the identity consisting of {\em normal}
subgroups.

  \begin{ex}{\label{E:prodiscrete1}} \end{ex}
    \begin{itemize}

     \item[$\mathbf{1.}$] For an arbitrary topological group $G$,
      there is a natural choice of  a prodiscrete topology on $G$,
      namely, the one where $\HH$ is the collection of all open-closed
      subgroups of $G$. We call this the {\bf canonical} prodiscrete
      topology on  $G$. There is another natural topology on $G$
      generated by {\em normal} open-closed subgroups. These two
      topologies are in general not the same.

     \item[$\mathbf{2.}$] A given discrete group $G$ can be endowed
      with several prodiscrete topologies. For instance, the {\em
      profinite, prosolvable, pronilpotent}, and {\em pro-$p$}
      topologies. The {\em discrete} topology is also  prodiscrete. In
      general, for any formation $\C$  of groups (that is, a
      collection of groups closed under taking quotients and finite
      sub-direct products), one can consider the pro-$\C$ topology on a
      group $G$. In the pro-$\C$  topology, $H\subseteq G$ is open if
      $H$ is normal and $G/H$ is in $\C$.

    All these topologies have a basis consisting of normal
    subgroups.

     \item[$\mathbf{3.}$] Given a family $\HH$ of subgroups of a
      group $G$, there is a smallest prodiscrete topology on $G$ in
      which every $H \in \HH$ is open. We call this the {\em topology
      generated by $\HH$}. Open subgroups in this topology are
      subgroups $K \subseteq G$ which contain some finite intersection
      of conjugates of groups in $\HH$. When $\HH$ consists of a single
      normal subgroup $N$, then the open subgroups of the topology
      generated by $N$ are exactly  the subgroups of $G$ which
      contain $N$.
    \end{itemize}

\vspace{0.1in}

Let $G$ be an arbitrary topological group. We define $\GSet$ to be
the category of all {\em continuous} discrete $G$-sets. There is a
forgetful functor   $F_G \:\GSet\to \mathsf{Set}$.

\vspace{0.1in}
\noindent{\bf Convention.} Throughout the paper all $G$-sets are assumed
to be continuous and discrete.

 \begin{defn}{\label{D:completion}}
      Let $G$ be an arbitrary topological group. We define the {\bf
      prodiscrete completion} of $G$ to be $G^{\wedge}:=\Aut(F_G)$,
      where $\Aut(F_G)$ is the group of self-transformations of the
      forgetful functor $F_G \:\GSet\to \mathsf{Set}$.
      When $G$ is given by an $\HH$-topology, we will also use the
      notation $G^{\wedge}_{\HH}$.
 \end{defn}

 \begin{rem}
  The prodiscrete completion  of a topological group  $G$ is the
  same as the  prodiscrete completion of $G$ endowed with its
  canonical prodiscrete topology (Example
  \ref{E:prodiscrete1}.$\mathbf{1}$). Therefore, it is natural to
  restrict to prodiscrete topological groups when talking about
  prodiscrete completions.
 \end{rem}

There is a natural homomorphism $\iota_G \: G \to G^{\wedge}$ which
is characterized by the property that, for every $X \in$ $\GSet$,
the actions of $G$ and $G^{\wedge}$ on $X$ are compatible with each
other via $\iota_G$. The map $\iota_G$ is injective if and only
if $G$ is Hausdorff.

We endow $G^{\wedge}$ with  a prodiscrete topology in which the open
subgroups are those subgroups  of $\Aut(F_G)=G^{\wedge}$ which are
of the form $U_{X,x}$, for $X \in$ $\GSet$ and $x \in X$. Here,
$U_{X,x}$  stands for the group of all elements in $\Aut(F_G)$ whose
action on $X$ leaves $x$ fixed.

 \begin{prop}{\label{P:topology}}
   The  subgroups $U_{X,x} \subseteq G^{\wedge}$, where $X \in$
   $\GSet$ and $x \in X$, are exactly the open subgroups
   of a prodiscrete topology  on $G^{\wedge}$.
 \end{prop}

 \begin{proof}
   We have to check the axioms $\mathbf{Top1}$,$\mathbf{2}$,$\mathbf{3}$.

  \vspace{0.1in}

   \noindent{\em Proof of} $\mathbf{Top1}$.
       This follows from the equality $U_{X,x}\cap U_{Y,y}=U_{X\x
       Y,(x,y)}$. The proof of this equality is easy, but there is a
       tiny subtlety. The point is that, there are, a priori, two
       actions of $G^{\wedge}$ on $X\x Y$. One is the componentwise action
       on the product. The other is obtained by considering $X\x Y$ as
       an object in  $\GSet$ and then taking the induced action of
       $G^{\wedge}=\Aut(F_G)$ on it; see Definition \ref{D:completion}.
       These two actions are, however,  identical, as can be seen by
       considering the two ($G$-equivariant) projection maps
       $\pr_1\: X\x Y \to X$ and $\pr_2\: X\x Y \to X$ and using the
       fact that every $\gamma \in G^{\wedge}$, being a
       transformation of functors, should respect $\pr_1$ and $\pr_2$.

   \vspace{0.1in}
   \noindent{\em Proof of} $\mathbf{Top2}$.
        For every $\gamma \in G^{\wedge}$
        and every $U_{X,x}$, we have
        $\gamma U_{X,x}\gamma^{-1}=U_{X,\gamma{x}}$.

   \vspace{0.1in}
   \noindent{\em Proof of} $\mathbf{Top3}$.
   Note that in the definition of $U_{X,x}$, we may assume that the
   action of $G$ on  $X$ is transitive (because $U_{X,x}=U_{G\cdot
   x,x}$). This implies that the action of $G^{\wedge}$ on $X$ is
   also transitive. Let $U \subseteq G^{\wedge}$ be a subgroup that
   contains $U_{X,x}$. We have to construct a $G$-set $Y$ and a
   point $y \in Y$ such that $U$ is exactly the stabilizer of
   $y$ in $G^{\wedge}$.

    Let $A:=U\cdot x \subseteq X$ be the orbit of $x$ under the
   action of $U$. It is easy to see that, for every
   $\gamma, \gamma' \in G^{\wedge}$, either $\gamma\cdot
   A=\gamma'\cdot A$ or $\gamma\cdot A\cap \gamma'\cdot
   A=\emptyset$; the equality happens exactly when $\gamma$ and
   $\gamma'$ are in the same left coset of $U$ in $G^{\wedge}$.
   Since the action of $G^{\wedge}$ on $X$ is transitive, this
   partitions $X$ into translates of $A$. Let $Y$ be the set of
   equivalence classes. (In other words, $Y$ is just the set
   $G^{\wedge}/U$.) We have an induced action of $G^{\wedge}$
   on $Y$, hence also one of $G$ on $Y$. This way, $Y$ becomes a
   $G$-set. Under the action of $G^{\wedge}$,  the stabilizer group
   of the class of $A$ in $Y$ is exactly $U$. This completes the
   proof.

  There is a tiny subtlety in the above argument that needs some
  explanation. Note that $Y$, viewed as an object of
  $\GSet$ inherits an action of $\Aut(F_G)=G^{\wedge}$
  which, a priori, may be different   from the original action of
  $G^{\wedge}$ on it. However, these two action are actually the same.
  This can be seen by looking at the projection map $p\: X \to Y$,
  viewed as a morphism in $\GSet$,  and using the fact
  elements of $G^{\wedge}=\Aut(F_G)$ are transformations of functors
  (hence respecting $p$).
 \end{proof}

The following lemma is immediate.

  \begin{lem}{\label{L:1}}
      The map $U \mapsto \iota_G^{-1}(U)$ induces a bijection
      between open subgroups  $U$ of $G^{\wedge}$ and open subgroups
      of $G$. This bijection sends normal subgroups to normal
      subgroups. For every open subgroup $U \subseteq G^{\wedge}$,
      we have a bijection $G/\iota_G^{-1}(U) \risom G^{\wedge}/U$.
      In particular, the map $\iota_G \: G \to G^{\wedge}$ is
      continuous. Indeed, the topology on $G^{\wedge}$ is the finest
      topology that makes $\iota_G$ continuous. Furthermore, the
      natural functor $\iota_G^* \: G^{\wedge}\mbox{-}\mathsf{Set} \to
      G\mbox{-}\mathsf{Set}$ is an equivalence of categories. Finally,
      $\iota_G^*$  respects the forgetful functors $F_G \:
      G\mbox{-}\mathsf{Set} \to \mathsf{Set}$ and $F_{G^{\wedge}} \:
      G^{\wedge}\mbox{-}\mathsf{Set} \to \mathsf{Set}$. That is, the diagram
               $$\xymatrix@R=12pt@C=0pt@M=10pt{
           G^{\wedge}\mbox{-}\mathsf{Set}
                     \ar[rr]^{\iota_G^*}\ar[dr]_{F_{G^{\wedge}}}
                            & & G\mbox{-}\mathsf{Set} \ar[ld]^{F_G}  \\
                                                 & \mathsf{Set} &     }$$
      is commutative.
    \end{lem}

The lemma implies that any continuous action of $G$ on a set $X$
extends uniquely to a continuous action of $G^{\wedge}$ on $X$, and
every $G$-equivariant map $X \to Y$ is automatically
$G^{\wedge}$-equivariant. The group $G^{\wedge}$ is universal among
the topological  groups that have the same  category of $G$-sets as
$G$.

   \begin{defn}{\label{D:complete}}
    We say that a prodiscrete group $G$ is {\bf complete} if
    $\iota_G \: G \to G^{\wedge}$ is an isomorphism.
   \end{defn}

   \begin{prop}{\label{P:complete}}
    For every  prodiscrete group $G$ the prodiscrete completion
    $G^{\wedge}$ is complete.
   \end{prop}

   \begin{proof}
      By Lemma \ref{L:1}, the category $\GWSet$
      is equivalent to $\GSet$ via an equivalence that respects
      the forgetful functors. This equivalence induces a natural isomorphism
      between the groups of self-transformations of the two forgetful
      functors.
   \end{proof}

   \begin{prop}{\label{P:discrete}}
      Every discrete group is complete.
   \end{prop}

   \begin{proof}
       Pick a self-transformation $\gamma \in \Aut(F_G)$  of the
       forgetful functor $F_G \:\GSet\to \mathsf{Set}$.
       We have to show that there exists $g \in G$ such that for
       every $G$-Set $X$ the action of $\gamma$ on $X$ is
       the same as the action of $g$. It is enough to assume that
       $X=G/H$, for some subgroup $H$ of $G$. (This is because if the
       actions of $g$ and $\gamma$ coincide for every transitive
       $G$-set, then they coincide for every $G$-set.) In fact,
       since $G \to G/H$ is surjective, it is enough to assume
       that $X=G$ with the left multiplication action.

       Let $G_{tirv}$ be the set $G$ with the trivial $G$-action.
       Since the action of $\gamma$ on a point is trivial, its
       action on $F_G(G_{tirv})$ is also trivial.
       By considering the action of $\gamma$ on the  diagram of $G$-sets
         $$\xymatrix@=16pt@M=8pt{
           G  &  G\x G_{triv} \ar[l]_(0.6){\pr_1}
           \ar[d]^{\text{mult.}} \ar[r]^(0.55){\pr_2} &  G_{triv} \\
              &    G    &           }$$
        it follows that the action of $\gamma$ respects right
        multiplication by every element of $G$. Therefore, $\gamma$
        must be equal to left multiplication by the element
        $g:=\gamma(1)\in G$.
   \end{proof}

   \begin{ex}{\label{E:prodiscrete2}}\end{ex}
     \begin{itemize}

       \item[$\mathbf{1.}$] Let $G$ be a group, and let $N \subseteq G$
         be a normal subgroup. Endow $G$ with the prodiscrete
         topology in which open subgroups are exactly the ones containing
         $N$. Then $G^{\wedge}\cong G/N$.

      \item[$\mathbf{2.}$] Let $G$ be a finite group. Then, every
        prodiscrete topology on $G$ is of the form above. To see this,
        take $N$ to be the intersection of  all  open subgroups of $G$.
     \end{itemize}
     \vspace{0.1in}

Given a continuous group homomorphism $f \: G \to H$, we have an
induced continuous homomorphism $f^{\wedge} \: G^{\wedge} \to
H^{\wedge}$ making the following square  commute
         $$\xymatrix{ G \ar[r]^{f} \ar[d]_{\iota_G}
                                       & H \ar[d]^{\iota_{H}}  \\
           G^{\wedge}  \ar[r]_{f^{\wedge}}  &     H^{\wedge}    }$$

  \begin{cor}{\label{C:extend}}
    Let $f \: G \to H$ be a continuous homomorphism,
    where $G$ is a prodiscrete group and $H$ is a discrete group.
    Then, there is a unique extension $f^{\wedge} \: G^{\wedge} \to H$.
  \end{cor}

  \begin{proof}
        By functoriality of $\iota$ we have a natural homomorphism
        $f^{\wedge} \: G^{\wedge} \to H^{\wedge}$. The
        assertion follows from Proposition \ref{P:discrete}.
  \end{proof}

Of course, the above statement is true for any complete
prodiscrete group $H$.

  \begin{lem}{\label{L:weak}}
       Let $G$ be a prodiscrete group, and let $G'$ be another
       prodiscrete group whose underlying group is the same as $G$
       but whose topology is weaker. Note that, by Lemma \ref{L:1},
       this induces a weaker  topology on $G^{\wedge}$, which we
       denote by $G''$. Then, the natural (continuous) homomorphism
       $G' \to G''$ induces an isomorphism $G'^{\wedge} \risom
       G''^{\wedge}$. (Note that when
       $G=G'$ we recover Proposition \ref{P:complete}.)
  \end{lem}

  \begin{proof}
       The categories   $G'\mbox{-}\mathsf{Set}$ and $G''\mbox{-}\mathsf{Set}$
       are naturally subcategories of $\GSet$ and
       $\GWSet$, respectively, and the
       restriction of  the equivalence $\iota_{G}^* \:
       G^{\wedge}\mbox{-}\mathsf{Set} \to G\mbox{-}\mathsf{Set}$ to
       $\GWSet$ induces an equivalence
       $G''^{\wedge}\mbox{-}\mathsf{Set} \to G'\mbox{-}\mathsf{Set}$
       (respecting the forgetful functors). This gives us the
       desired isomorphism $G'^{\wedge} \risom G''^{\wedge}$.
  \end{proof}

  \begin{rem}{\label{R:dense}}
     It seems that it is not true in general that the image of
     $\iota_G \: G \to G^{\wedge}$ is dense, unless we assume
     that $G$ has a basis consisting of {\em normal} open subgroups
     (see Proposition \ref{P:classicalcompletion}).
     All we can say in general is that there
     is no proper open subgroup of $G^{\wedge}$ containing $\iota_G(G)$.
  \end{rem}

  \begin{prop}{\label{P:classicalcompletion}}
     Assume $G$ is a prodiscrete group that has a basis
     $\N=\{N_i\}_{i \in I}$ (around identity) consisting of normal
     subgroups. Then, there is a natural isomorphism $$G^{\wedge}
     \cong \underset{\N}\liminv G/N_i.$$ In other words, in this
     case our notion of completion coincides with the classical one.
  \end{prop}

  \begin{proof}
        Denote the right hand side by $\tilde{G}$. It is easy to
        check  that the map $G \to \tilde{G}$ induces an
        equivalence of categories $\tilde{G}\mbox{-}\mathsf{Set}$ $\to
        G\mbox{-}\mathsf{Set}$ (respecting the forgetful functors $F_G$
        and $F_{\tilde{G}}$). So, it is enough to show that
        $\tilde{G}$ is complete, i.e.,
        $\iota_{\tilde{G}} \: \tilde{G} \to \Aut{F_{\tilde{G}}}$ is an
        isomorphism.

        This map in injective, since $\tilde{G}$ is Hausdorff. To
        prove the surjectivity, let $\al \in \Aut{F_{\tilde{G}}}$.
        Fix an $i \in I$ and consider the $\tilde{G}$-set $G/N_i$.
        The action of $\al$ on $G/N_i$ sends $1 \in G/N_i$ to some
        $g_i \in G/N_i$. For any $h \in G/N_i$, multiplication on
        the right by $h$ induces a map of $\tilde{G}$-sets $G/N_i
        \to G/N_i$. Therefore, since the action of $\al$ is
        functorial, its effect  on $G/N_i$ sends $h=1h$ to
        $g_ih$. That is, $\al$ acts by multiplication on the right by
        $g_i$.  Again, by the functoriality of the action of $\al$,
        the various $g_i$ are compatible, that is, they come from an
        element $g \in \underset{\N}\liminv G/N_i=\tilde{G}$.
        This proves the surjectivity.
  \end{proof}

\subsection{Review of Grothendieck's Galois theory}{\label{SS:Grothendieck}}

We review (a slightly modified version of) Grothendieck's theory
of Galois categories \cite{SGA1}. The difference here is that we
want to apply the theory to the cases where the covering maps
are not necessarily finite, so we will remove certain
finiteness assumptions.

   \begin{defn}[see \oldcite{SGA1} for more details]{\label{D:Galois}}
     By a {\bf Galois category} we mean a pair $(\mathsf{C},F)$, where
     $\mathsf{C}$ is a category and $F \: \mathsf{C} \to \mathsf{Set}$
     is a functor, satisfying the following axioms:\footnote{The axioms are
     numbered in this way to be compatible with \oldcite{SGA1}.}

    \begin{itemize}

     \item[$\mathbf{G1}.$] The category $\mathsf{C}$ has finite
      projective limits (i.e., $\mathsf{C}$ has a final object and
      fiber products exist).

     \item[$\mathbf{G2}.$]  Direct sums (not necessarily finite)
      exist. In particular, an initial object exists.  Also, quotient
      of an object under an equivalence relation exists. In
      particular, quotients under (faithful) group actions exist.

     \item[$\mathbf{G3}.$] Let $u \: X \to Y$ be a morphism in
      $\mathsf{C}$. Then, $u$ factorizes as $X \llra{u'} Y' \llra{u''}
      Y$, with $u'$ a strict epimorphism and $u''$ a  monomorphism
      that is an isomorphism onto a direct summand of $Y$.

     \item[$\mathbf{G4}.$] The functor $F$ is left exact. That is,
      it commutes with fiber products
      and takes the final object to the final object.

     \item[$\mathbf{G5}.$] The functor $F$ commutes with direct sums,
      takes strict epimorphisms to epimorphisms  and commutes with
      taking quotients (as in $\mathbf{G2}$).

     \item[$\mathbf{G6}.$]The functor $F$ is conservative. That is,
      if $u \: X \to Y$ is  a morphism in $\mathsf{C}$
      such that $F(u)$ is an isomorphism, then $u$ is an isomorphism.
    \end{itemize}
  \end{defn}

The functor $F$ is called the {\bf fundamental functor}. A functor
between Galois categories is called a {\bf Galois functor} if it
respects the fundamental functors. An object in a Galois category
is called {\bf connected} if it can not be written as a direct sum
of two objects. An example of a Galois category is the category of
continuous $G$-sets, where $G$ is an arbitrary topological group.
The fundamental functor in this case is the forgetful functor. The
main theorem of Grothendieck's Galois theory is that this is basically
the only example.

  \begin{thm}{\label{T:Galois}}
   Let $(\mathsf{C},F)$ be a Galois category. Let
   $\pi'_1(\mathsf{C},F):=\Aut F$ be the (complete prodiscrete) group
   of automorphisms of $F$. Then, there is a natural equivalence of
   Galois categories $\mathsf{C} \cong$
   $\pi'_1(\mathsf{C},F)\mbox{-}\mathsf{Set}$.\footnote{The reason for
   using the notation $\pi'_1$ becomes clear when we consider the
   Galois category associated to a topological space $X$
   ($\S$\ref{SS:Galois1}), in which case $\pi'_1$ and $\pi_1$ will
   not in general be the same unless we assume that
   $X$ is semilocally 1-connected.}
  \end{thm}

The above equivalence is functorial with respect to functors between
Galois categories. In other words, the category of Galois categories
and Galois functors between them is equivalent to the category of
complete prodiscrete groups and continuous group
homomorphisms.

The proof of the above theorem is just a slight modification of
the proof given in \cite{SGA1} and we omit it.

An object $X$ in a Galois category $(\mathsf{C},F)$ is called {\bf
Galois} if $X/\Aut X=*$. This means that the group of automorphisms
of $X$ (which we think of as the {\em Galois group} of $X$) is as
big as it can be. For example, in $\GSet$ every $G/N$, where $N$ is
an open normal subgroup of $G$, is a Galois object (and every
connected Galois object is of this form). In general, connected
Galois objects are in bijection with open
normal subgroups of $\pi'_1(\mathsf{C},F)$.

Note that, in Theorem \ref{T:Galois} we did not require the
existence of Galois objects in $\mathsf{C}$, although it will be the
case in most examples. In fact, in most situations, one can find a
cofinal family of connected Galois objects (i.e., every connected
object is dominated by a connected Galois object). This is
equivalent to saying that $\pi'_1(\mathsf{C},F)$ has a basis at the
identity consisting of normal subgroups. In this situation,
Proposition \ref{P:classicalcompletion} implies that
$\pi'_1(\mathsf{C},F)$ can be computed as the opposite of the
inverse limit of the Galois groups of the Galois objects. (The
reason we have to take the opposite is that the group of
automorphisms of the object $G/N \in$ $\GSet$ is
$(G/N)^{op}$.) Let us summarize this in the following proposition.

  \begin{prop}
    Let $(\mathsf{C},F)$ be a Galois category, and let $\{X_i\}_{i
    \in I}$ be a cofinal family of connected Galois objects (i.e.,
    every connected object is dominated by some $X_i$). Then, we have
    a natural isomorphism
         $$\pi'_1(\mathsf{C},F)  \cong
                     \underset{I}\liminv(\Aut X_i)^{op}. $$
  \end{prop}

\section{Galois theory of covering stacks}{\label{S:Galois}}

The Galois theory of covering stacks of a topological stack $\X$ is
closely related to the group theory of its fundamental group. But,
these two theories can diverge unless we assume that $\X$ behaves
nicely locally: the mouthful `semilocally 1-connected' condition.
This property, unfortunately, may not be preserved under certain
natural constructions that one makes with topological stacks (say,
base extension, or passing to the coarse moduli space). To avoid
this nuisance we begin by developing a theory of fundamental groups
which is more in tune with Galois theory of covering stacks. We then
explain how it relates to the usual theory of fundamental groups.

\subsection{Review of topological stacks}{\label{SS:topological}}
We recall a few definitions from \cite{homotopy}. We follow the
notational convention of [ibid.] by using calligraphic symbols
$\X$,$\Y$,$\Z$,... for stacks and script symbols $X$,$Y$,$Z$,... for
spaces.

Throughout the paper, a stack means a stack over the site
$\mathbf{Top}$ of topological spaces; here, $\mathbf{Top}$ is
endowed with its standard open-cover topology. A stack $\X$ is
called {\bf topological} if it is equivalent to the quotient stack
of a topological groupoid $[X_1\sst{}X_0]$ whose source maps is a
{\em local Serre fibration} in the sense of \cite{homotopy},
$\S$13.1. An {\em atlas}\footnote{In \oldcite{homotopy} we call this
a chart. Bad terminology!} for a topological stack $\X$ is an
epimorphism $p \: X \to \X$ from a topological space $X$. Given such
an atlas, one finds a groupoid presentation for $\X$ by taking
$X_1=X\times_{\X}X$ and $X_0=X$.

Every {\em topological space} $X$ can be thought of as a topological
stack, namely, the topological stack associated to the trivial
groupoid $[X \sst{} X]$. This gives a fully faithful embedding of
the category of topological spaces and continuous maps to the
category of topological stacks. Also, for every topological group $G$,
every continuous $G$-space gives rise to a
topological stack $[X/G]$.  Among other examples of topological
stacks we mention   {\em orbifolds}, the underlying topological
stacks of  {\em Artin stacks}, {\em complexes-of-groups}, and {\em
leaf stacks} of foliations.

The basic notions of algebraic topology (e.g., homotopy, homotopy
groups, generalized homology/cohomology, fibrations,
mapping spaces, loop spaces, etc.) generalize to topological stacks.

By a {\em point} $x$ in a stack $\X$   we mean   a morphism $x \: *
\to \X$ of stacks. The {\em inertia} group of a point $x$ is the
group of self-transformations of the above map. It is naturally a
topological group when $\X$ is topological, and is denoted by $I_x$.

To a topological stack $\X$ one associates a topological space
$\Xm$, called the {\bf coarse moduli space} of $\X$. There is a
natural morphism $\pi \: \X \to \Xm$ that is universal among
morphisms from $\X$ to topological spaces. The map $\pi$ induces a
natural bijection between the set of 2-isomorphism classes of points
of $\X$ and the set of points of  $\Xm$. When $\X=[X_0/X_1]$ for a
topological groupoid $[X_1\sst{}X_0]$, $\Xm$ is naturally
homeomorphic to the coarse quotient space $X_0/X_1$. In particular,
when $\X=[X/G]$ is the quotient stack of a group action, the coarse
moduli space $\Xm$ is homeomorphic to the coarse quotient space
$X/G$.

\subsection{Review of covering stacks}{\label{SS:covering}}

We review a few basic facts about covering stacks of topological
stacks. More details and proofs can be found in \cite{homotopy},
$\S$18.

 \begin{defn}{\label{D:pc}}
      Let $\X$ be a topological stack. We say that $\X$ is {\bf
      connected} if it has no proper open-closed substack. We say
      $\X$ is {\bf path-connected}, if for every two points $x$ and
      $y$ in $\X$,   there is a path from $x$ to $y$.
 \end{defn}

 \begin{defn}{\label{D:lpc}}
      Let $\X$ be a topological stack. We say that $\X$ is {\bf locally
      connected} (respectively, {\bf locally path-connected}, {\bf
      semilocally 1-connected}),   if there is an atlas $X \to \X$
      such that  $X$ is so.
 \end{defn}

These definitions agree with the usual definitions when $\X$ is a
topological  space. This is because of the following lemma.

 \begin{lem}{\label{L:lpc}}
    Let $f \: Y \to X$ be a continuous map of topological spaces
    that admits local sections. Assume $Y$ is locally connected
    (respectively, locally path-connected, semilocally 1-connected).
    Then so is $X$.
 \end{lem}

 \begin{defn}{\label{D:covering}}
    A representable map $\Y \to \X$ of  topological stacks is called
    a {\bf covering map} if  for every topological space $W$ and
    every map $W \to \X$, the base extension
    $W\x_{\X}\Y \to W$ is a covering map of topological spaces.
 \end{defn}

 \begin{prop}{\label{P:diagonal}}
    Let $f \: \Y \to \X$ be a covering map of topological stacks.
    Then, the diagonal map $\De \: \Y \to \Y\x_{\X} \Y$ is an
    open-closed embedding.
 \end{prop}

An immediate corollary of this proposition is the following.

 \begin{cor}{\label{C:openclosed}}
    Let $f \: \Y \to \X$ be a covering map of topological stacks.
    Let $p\: X \to \X$ be an atlas for $\X$, and let $q \: Y \to \Y$
    be the pull-back atlas for $\Y$, where $Y=\Y\x_{\X}X$. (Note that
    we can also view $Y$ as an atlas for $\X$ via $f\circ q\: Y \to
    \X$.) Set $R=Y\x_{\Y}Y$ and $R'=Y\x_{\X}Y$, and consider the
    corresponding groupoids $[R\sst{}Y]$ and $[R'\sst{}Y]$ (so
    $[Y/R]\cong\Y$ and $[Y/R']\cong\X$). Then, $[R\sst{}Y]$ is an
    open-closed subgroupoid of $[R'\sst{}Y]$.
 \end{cor}

\subsection{$\mathsf{C}$-complete fundamental groups}{\label{SS:Fundamental}}

In this subsection we look at the topological incarnations of the
notions developed in the previous section. Let $\X$ be a connected
locally path-connected topological stack, and let $x$ be a point in
$\X$. We will not assume yet that $\X$ is semilocally 1-connected.

To $(\X,x)$ we can associate various Galois categories
$(\mathsf{C},F)$ of covering stacks  by requiring that

  \begin{itemize}
      \item[$\mathbf{C1.}$]  If $\Y$ is in $\mathsf{C}$ and $\Y'$
        is another covering stack of $\X$ that is dominated by $\Y$
        (i.e., there is a surjection $\Y \to \Y'$ relative to $\X$),
        then $\Y'$ is in  $\mathsf{C}$.

      \item[$\mathbf{C2.}$] $\mathsf{C}$ is closed under fiber products.

      \item[$\mathbf{C3.}$] $\mathsf{C}$ is closed under  taking
        disjoint unions.

      \item[$\mathbf{C4.}$] If $\Y$ is in $\mathsf{C}$ , then every
        connected component of $\Y$ is also in $\mathsf{C}$.
  \end{itemize}

\noindent The fundamental functor $F$ for such a category is simply
the fiber functor $\Y \mapsto \Y_x=*\x_x \Y$. Axioms
$\mathbf{G1}\mbox{-}\mathbf{G6}$ are easy to verify. (Perhaps
$\mathbf{G2}$ is a bit non-trivial. It follows from Lemma
\ref{L:G2} below.)

For a given (connected) covering stack $f \: \Y \to \X$, and for a
choice of a base-point $y$ in $\Y$ lying above $x$ (i.e., $y$ is a
point in the fiber $\Y_x$ of $\Y$ over $x$), we have an injection of
fundamental groups $\pY \to \pX$. The image of this injection
uniquely determines $(\Y,y)$ up to isomorphism.
Changing the base point $y$ will change this subgroup by conjugation.

Let $\HH_{\mathsf{C}}$ be the collection of all such subgroups
of $\pX$. It is easily seen that the axioms  $\mathbf{Top1,2,3}$ of
$\S$\ref{SS:pro} are satisfied:  $\mathbf{Top1}$ follows from
$\mathbf{C2}$; $\mathbf{Top2}$ follows from the discussion of the
previous paragraph about changing the base point; $\mathbf{Top3}$
follows from $\mathbf{C1}$ and Corollary \ref{C:Top3} below.

 \begin{lem}{\label{L:G2}}
   Let $(\X,x)$ be a pointed topological stack, and let $f \: \Y
   \to \X$ be an arbitrary covering stack of $\X$.

    \begin{itemize}

      \item[$\mathbf{i.}$] Let $\Z$ be a covering stack of $\X$, and
       let $\Z \to \Y\x_{\X}\Y$ be an equivalence relation on $\Y$
       (see Axiom $\mathbf{G2}$). Then, the quotient $\Y'$ of this
       equivalence relation exists and is a covering stack of $\X$.

      \item[$\mathbf{ii.}$] Let $F'$ be a $\pX$-set, and let
       $F_x(\Y) \to F'$ be a surjective $\pX$-equivariant map, where
       $F_x(\Y)$ is the fiber of $f$ over $x$. Then, there exists a
       unique covering stack $\Y' \to \X$ of $\X$ whose fiber is
       isomorphic to $F'$ (as a $\pX$-set), together with a map   of
       covering stacks $\Y \to \Y'$ realizing  $F_x(\Y) \to F'$.

   \end{itemize}

  \end{lem}

  \begin{proof}[Proof of {\em ($\mathbf{i}$)}]
      Choose an atlas  $X \to \X$ for $\X$. Let $Y \to \Y$ and $Z
      \to \Z$ be the pull-back atlases for $\Y$ and $\Z$. Let
      $[R_X \sst{} X]$, $[R_Y \sst{} Y]$ and $[R_Z \sst{} Z]$ be the
      corresponding groupoids. Then, $Z \to Y\x_{X} Y$ and $R_Z \to
      R_Y\x_{R_X} R_Y$ are equivalence relations. Also, note that
      the maps $Y\to X$, $Z\to X$, $R_Y \to R_X$ and $R_Z\to R_X$
      are all covering maps (each being the base extension of either
      $\Y \to \X$ or $\Z \to \X$). Set $Y':=Y/Z$ and
      $R_{Y'}:=R_Y/R_Z$. It is easy to see that we have a natural
      groupoid structure on $[R_{Y'} \sst{} Y']$. The quotient stack
      $\Y':=[Y'/R_{Y'}]$ is the desired quotient of $\Y$ by $\Z$.

  \vspace{0.1in}
  \noindent {\em Proof of}  ($\mathbf{ii}$).
      Note that the statement is true for topological spaces.  So,
      as in the previous part, by choosing an atlas  $X \to \X$
      we will reduce the problem to the case of topological
      spaces.  Let $Y$, $R_X$ and $R_Y$ be as in the previous part.
      Let $x_0\in X$ be a lift of $x$, and let $x_1 \in R_X$ be the
      corresponding point in $R_X$. Note that the maps $Y \to X$ and
      $R_Y \to R_X$ are base extensions of $f$, so both are covering
      maps. Furthermore, the fibers $F_{x_0}(Y)$ and $F_{x_1}(R_Y)$
      are, as sets, in natural bijection  with $F_x(\Y)$. The
      actions of $\pi_1(X,x_0)$ and $\pi_1(R_X,x_1)$ on these sets
      are obtained from that of $\pX$ on $F_x(\Y)$ via the group
      homomorphisms $\pi_1(X,x_0) \to \pX$ and $\pi_1(R_X,x_1) \to
      \pX$, respectively. We are now reduced to the case of
      topological spaces, with $\X$ replaced by $X$ and $R_X$,
      respectively. So, we can construct a covering space $Y'$
      of $X$ and a covering space  and $R_{Y'}$ of  $R_X$, together
      with maps $Y \to Y'$ and $R_Y\to R_{Y'}$.
      It is easy to see that there is a natural  groupoid structure
      on $[R_{Y'} \sst{} Y']$. The quotient stack $\Y':=[Y'/R_{Y'}]$
      is the desired covering stack
      of $\X$.
  \end{proof}

  \begin{cor}{\label{C:Top3}}
      Let $(\X,x)$ be a connected pointed topological stack. Let
    $f \: (\Y,y) \to (\X,x)$ be a connected covering stack, and let $H
    \subseteq \pX$ be the corresponding subgroup. Let $H' \subseteq
    \pX$ be a subgroup containing $H$. Then, there exists a
    (pointed) covering stack $\Y'$ of $\X$ corresponding to $H'$.
  \end{cor}

We summarize our discussion by saying that   $\HH_{\mathsf{C}}$ is
exactly the set of open subgroups of a prodiscrete topology
on $\pX$.  The category $\mathsf{C}$ is  equivalent to the category
of continuous $\pX$-sets. If we denote by ${\pX}_{\mathsf{C}}$  the
completion of $\pX$ (Definition \ref{D:completion}) with respect to
this topology, we have proved the following

  \begin{prop}
    The notation being as above, there is  an equivalence of
    Galois categories $(\mathsf{C},F) \cong$
    ${\pX}_{\mathsf{C}}\mbox{-}\mathsf{Set}$.
  \end{prop}

We call ${\pX}_{\mathsf{C}}$ the $\mathsf{C}$-{\bf complete
fundamental group} of $(\X,x)$.

\section{Main examples of Galois
categories}{\label{S:GaloisExamples}}

We list the main examples of Galois categories that we are
interested in. More examples can be produced by noting that, if
$\mathsf{C}$ and $\mathsf{C}'$ satisfy the above axioms, then so
does their ``intersection'' $\mathsf{C}\cap\mathsf{C}'$. Here, by
$\mathsf{C}\cap\mathsf{C}'$ we mean the category of all covering
stacks $\Y \to \X$ that are isomorphic to a covering stack in
$\mathsf{C}$ and a covering stack in $\mathsf{C}'$.

\subsection{The Galois category $\mathsf{Full}$}{\label{SS:Galois1}}

The category $\mathsf{Full}$ of all covering stacks of $\X$
satisfies the required axioms. This gives the finest
$\mathsf{C}$-topology on $\pX$, which we call the {\em full}
topology. The corresponding prodiscrete fundamental group
${\pX}_{\mathsf{Full}}$ is denoted by $\pi_1'(\X,x)$. There is a
natural homomorphism $\pX \to \pi_1'(\X,x)$. This map is not in
general an isomorphism. (This is due to the fact that the universal
cover may not exist; see Example \ref{E:Hawaiian} below). This
homomorphism is, however,  an isomorphism if $\X$ is connected,
locally path-connected, and semilocally
1-connected, because in this case the universal cover exists
(\cite{homotopy}, $\S$18) and the full topology is discrete.

\begin{ex}[Hawaiian ear-ring]{\label{E:Hawaiian}}
  Let $X$ be the subspace of $\mathbb{R}^2$ defined as the union of
  circles of radius $2^{-n}$, $n\in \bbZ$, centered at the points
  $(2^{-n},0)$. This space is not semilocally 1-connected because
  every open set containing the origin contains all but finitely
  many of the circles.

  Let $Y=\bigvee_{i\in \bbZ} S^1$ be the wedge sum of countably many
  circles. This is a semilocally 1-connected space. We have
  $\pi_1(Y)\cong \mathbb{F}^{\bbZ}$, the free group on countably
  many generators $\{a_i \ |\ i\in \bbZ\}$. Consider the prodiscrete
  topology on $\pi_1(Y)$ generated by the subgroups $H_n$, where
  $H_n$ is the normal subgroup generated by $\{a_i \ |\    i>n \}$.
  Let $\pi_1(Y)^{\wedge}$ be the prodiscrete completion.

  There is a natural continuous bijection $f \: Y \to X$. This map
  induces an isomorphism $\pi_1(Y)^{\wedge} \to \pi'_1(X)$. More
  explicitly,
    $$\pi'_1(X)\cong \liminv_{n \in \bbZ} \mathbb{F}^{\bbZ}/H_n\cong
       \liminv_{n \in \bbN} \mathbb{F}^{\leq n},$$
  where $F^{\leq n}$ is the free group on generators $\{a_i \
  |\   i \leq n  \}$ and $F^{\leq n+1} \to F^{\leq n}$ is the map
  that sends $a_{n+1}$ to $1$.
 \end{ex}

\subsection{The Galois category $\mathsf{Fin}$}{\label{SS:Galois2}}

Another example is the category $\mathsf{Fin}$ of covering stacks of
$\X$ each of whose connected components has finite degree over $\X$.
The corresponding topology on $\pX$ consists of open sets of the
full topology which have finite index in $\pX$. This topology has a
basis of normal subgroups. Therefore, the corresponding
completion can be computed (using Proposition
\ref{P:classicalcompletion}) as

            $${\pX}^{\wedge}_{\mathsf{Fin}} \cong \underset{\N}\liminv
                                       \pX/N,$$
\noindent where $\N$ is the set of all open (in the full topology)
normal subgroups of finite index of $\pX$. When $\X$ is semilocally
1-connected (so full topology is discrete) this coincides with the
profinite completion of  $\pX$.

\subsection{The Galois category $\mathsf{FPR}$}{\label{SS:Galois3}}

  \begin{defn}{\label{D:FPR}}
   Let $f \: \Y \to \X$ be a representable morphism of stacks. Let
   $y \in \Y$ be a point in $\Y$. We say that $f$ is {\bf
   fixed-point-reflecting} (FPR for short) at $y$ if the induced map
   $I_y \to I_{f(y)}$ (which is a priori injective) is an
   isomorphism. We say that $f$ is FPR  if it is FPR at every point.
  \end{defn}

The following lemma will be used in $\S$\ref{SS:FreeFPR}.

  \begin{lem}{\label{L:FPR}}
    Let $f\: \Y \to \X$ be a representable morphism of stacks, and
    let $x \in \X$ be a point. Let $\Y_x$ be the fiber of  $f$ over
    $x$. Assume $\Y_x$ is finite. Let $\bar{x} \in \Xm$   be the
    image of $x$ in $\Xm$, and let $\Y_{mod,x}$ be the fiber of
    $f_{mod} \: \Ym \to \Xm$ over $\bar{x}$. Then $\#(Y_{mod,x})\leq
    \#(\Y_x)$. The equality holds if and only if
    $f$ is FPR at every point in the fiber of $\Y$ over $x$.
  \end{lem}

  \begin{proof}
   Easy.
  \end{proof}

It is easy to check that the category $\mathsf{FPR}$ of all FPR
covering stacks of $\X$ satisfies the axioms
$\mathbf{C1}\mbox{-}\mathbf{C4}$ of $\S$\ref{SS:Fundamental}.

\subsection{The Galois category $\mathsf{Free}$}{\label{SS:Galois4}}

Let   $\mathsf{Free} \subseteq \mathsf{FPR}$ be the category of all
FPR   covering stacks $\Y \to \X$ such that the induced map $\Ym \to
\Xm$ is again a covering map. Then, $\mathsf{Free}$ satisfies the
axioms. We will only prove this in a special case where  $\X$ is
strongly locally path-connected (Definition \ref{D:stronglylc}), in
which case we will show that  $\mathsf{Free}=\mathsf{FPR}$
(Proposition \ref{P:strongFPR}).

We will also see in Proposition \ref{P:properFPR} that, under some
mild conditions on $\X$, every finite FPR covering stack is
automatically free. That is,
$\mathsf{FPR}\cap \mathsf{Fin}  =\mathsf{Free}\cap \mathsf{Fin}$.

\section{Topological stacks with slice property}{\label{S:Mild}}

In this section, we introduce an important class of topological
stacks which behave particularly well locally. We call these {\em
topological stacks with slice property}. The slice property is the
key in proving the main theorems of the paper.

\vspace{0.1in}
\noindent {\bf Notation.} Let $G$ be a topological group acting on a
space $X$, and let $x$ be a point in $X$. By a slight abuse of
notation, we will denote the stabilizer group of $x$ by $I_x$.
We will  view $I_x$ as a topological group.

 \begin{defn}[\oldcite{Palais}, $\S$2.1]{\label{D:mild}}
   Let $G$ be a topological group acting continuously on a
   topological space $X$, and let $x$ be a point in $X$.
   A subset $S$ of $X$ is called a {\bf slice
   at} $x$ if   it has the following properties:
   \begin{itemize}
    \item[$\mathbf{S1}$.] The subset $I_xS \subseteq X$ is open and
     there exists a $G$-equivariant map $f \: GS \to G/I_x$ whose
     fiber over the point $I_x \in G/I_x$ is precisely $S$.

    \item[$\mathbf{S2}$.] There exists an open subspace $U \subseteq
     G/I_x$ and a local section $\chi \: U \to G$ such that $(u,s)
     \mapsto \chi(u)s$ is a homeomorphism of $U\times S$ onto an
     open neighborhood of $x$ in $X$.
   \end{itemize}

   We say that the action has {\bf  slice property at} $x$ if every
   open neighborhood of $x$ contains a slice at $x$. We say that a
   group action  has slice property if it has slice property at every
   point.
 \end{defn}

 \begin{rem}{\label{R:nearslice}}
 \end{rem}
   \begin{itemize}

    \item[$\mathbf{1}$.] It follows from (\cite{Palais}, Proposition
     2.1.3) that the natural $G$-equivariant map $G\times_{I_x} S \to
     GS$ is a homeomorphism. Here, $G\times_{I_x} S$ is the quotient of
     $G\times S$ under the action of $I_x$ defined by
     $\alpha\cdot(g,x)=(g\alpha^{-1},\alpha x)$.

    \item[$\mathbf{2}$.] In the case where $G$ is a Lie group (not
     necessarily compact), ($\mathbf{S2}$) follows from
     ($\mathbf{S1}$). This is \cite{Palais}, Proposition 2.1.2.

    \item[$\mathbf{3}$.] In the case where $I_x$ is compact, existence
     of a slice at $x$ implies existence of slices that are arbitrarily
     small. Therefore, when the action of $G$ on $X$ has compact
     stabilizers, to check whether the action has slice property it is
     enough to check that there exists at least one slice at every point.
   \end{itemize}

 \begin{lem}{\label{L:nearslice}}
   Let $G$ be a topological group acting on a topological space $X$.
   Let $x$ be a point in $X$ and $S$ a subset containing $x$. Then,
   $S$ is a slice at $x$ if and only if the  map
   of stacks $[S/I_x] \to [X/G]$ is an open embedding.
 \end{lem}

 \begin{proof}
  We only prove the `only if' part which is what we need in the rest
  of the paper. We will show that the map $[S/I_x] \to [X/G]$
  identifies $[S/I_x]$ with the open substack $[GS/G]$ of $[X/G]$.

  Consider $G\times S$ endowed with the $G\times I_x$ action defined
  by $(g,\alpha)\cdot(h,x):=(gh\alpha^{-1},\alpha x)$. Let $\varphi
  \: G\times I_x \to I_x$ and $f \: G \times S \to S$ be the
  projection maps. It is clear that the $\varphi$-equivariant map $f$
  induces an equivalence of quotient stacks $[(G \times S)/( G\times
  I_x)] \to [S/I_x]$. So, it is enough to show that the map
  $[(G \times S)/( G\times I_x)] \to [GS/G]$ is an equivalence of
  stacks. This map can be written as a composition $[(G \times S)/(
  G\times I_x)] \to [(G\times_{I_x}S)/G] \to [GS/G]$. The first map
  is obviously  an equivalence of stacks. The second map is an
  equivalence of stacks by virtue of Remark
  \ref{R:nearslice}.$\mathbf{1}$.
 \end{proof}

 \begin{defn}{\label{D:strong}}
  We say that a topological stack $\X$ has   {\bf  slice property},
  if for every point $x$ in $\X$ and every open substack $\U
  \subseteq \X$ containing $x$, there is an open substack $\V
  \subseteq \U$ such that $\V \cong [V/I_x]$, where $V$ is a
  topological space with an action of $I_x$ that has slice property
  at $x$. In the case where $I_x$ are discrete groups, such stacks
  are called {\bf Deligne-Mumford} topological
  stacks in \cite{homotopy}, $\S$14.
 \end{defn}

 \begin{lem}{\label{L:slice2}}
  Let  $\X$ be a topological stack that can be covered by open
  substacks of the from $[X/G]$, where $G$ is a topological group
  acting on $X$ with slice  property.  Then $\X$ has slice property.
 \end{lem}

 \begin{proof}
  Trivial.
 \end{proof}

\subsection{Examples of stacks with slice property}{\label{SS:slice}}

We list some general classes of group actions with slice
property.

 \begin{itemize}
   \item[$\mathbf{1.}$] The continuous action of a finite group on a
    topological space has slice property.

  \item[$\mathbf{2.}$] Let $G$ be a Lie group
   (not necessarily compact) acting on a topological space $X$.
   Assume $X$ is a {\bf Cartan} $G$-space in the sense of
   (\cite{Palais}, Definition 1.1.2). Then, the action has slice
   property (\cite{Palais}, theorem 2.3.3). We recall from [ibid.]
   that $X$ is called a Cartan $G$-space if for every point of $X$
   there is an open neighborhood $U$ such that the set $\{g \in G \,
   | \, gU\cap U\neq\emptyset\}$ has compact closure. For instance,
   if $X$ is locally compact and the action is proper (i.e.,
   $G\times X \to X\times X$ is a proper map), then $X$ is a Cartan
   $G$-space. Also, if $X$ is completely regular and $G$
   compact Lie, then $X$ is a Cartan $G$-space.
 \end{itemize}

The following proposition is immediate.

 \begin{prop}{\label{P:strong1}}
  Let $\X$ be a topological stack (which can be covered by open
  substacks) of the form $[X/G]$ with $X$ a Cartan $G$-space.
  Then, $\X$ has slice property.
 \end{prop}

  \begin{lem}{\label{L:opensubstack}}
    Let $[R\sst{}X]$ be a topological groupoid. Assume that the
    source map $s\: R \to X$ is open and has the property that for
    every open $V \subseteq R$ the induced map $s|_V \: V \to s(V)$
    admits local sections. Then, for every open $U \subseteq X$, the
    induced map $[U/R|_U] \to [X/R]$ is an open embedding. Here,
    $[R|_U \sst{} U]$ stands for the restriction of
    $R$ to $U$, which is defined by $R|_U=(U\x U)\x_{X\x X}R$.
   \end{lem}

  \begin{proof}
     The map $[U/R|_U] \to [X/R]$ is always a monomorphism (i.e.,
     fully faithful), without any assumptions on the source map $s$.
     The extra assumption on $s$  implies that $[U/R|_U]$ is
     equivalent to the open substack $[\OO(U)/R|_{\OO(U)}]$ of
     $[X/R]$, where $\OO(U)=t\big(s^{-1}(U)\big)$ is the orbit of
     $U$ (which is open).
  \end{proof}

Every groupoid $[R\sst{}X]$ in which the source map $s \: R\to X$ is
locally isomorphic to the projection map $Y\x X\to X$of a product
has the property required in Lemma \ref{L:opensubstack}. These
include action groupoids   of  topological groups acting
continuously on topological spaces. Lie groupoids
also have this property.

 \begin{prop}{\label{P:strong2}}
  Let $\X=[X/R]$, where $[R\sst{}X]$ is a proper Lie groupoid.
  Then, $\X$ has slice property.
 \end{prop}

 \begin{proof}
  Let $x \in X$ be an arbitrary point. By (\cite{Weinstein},
  Proposition 2.4), the orbit $\OO(x)$ is a closed submanifold of
  $X$. Choose a small enough transversal $\Sigma$ to the orbit
  $\OO(x)$ at $x$. The map $t|_{s^{-1}\Sigma} \: s^{-1}\Sigma \to X$
  is a submersion, so the quotient stack of the restriction groupoid
  $[R|_{s^{-1}\Sigma} \sst{} \Sigma]$ is an open substack of $\X$
  by Lemma \ref{L:opensubstack}. Now, $[R|_{s^{-1}\Sigma} \sst{}
  \Sigma]$ is a Lie groupoid that has $x$ as a fixed point. So, by
  (\cite{Zung}, Theorem 1.1), we can shrink $\Sigma$ (as small as we
  want) and assume that $[R|_{s^{-1}\Sigma} \sst{} \Sigma]$  is
  isomorphic to a (linear) action groupoid of the stabilizer group $I_x$.
 \end{proof}

\section{Strongly locally path-connected topological
stacks}{\label{S:Strongly}}

We begin with a definition.

 \begin{defn}{\label{D:stronglylc}}
   A topological stack $\X$ is {\bf strongly locally path-connected}
   if it has slice property and, furthermore,
   the topological spaces $V$ of Definition \ref{D:strong}
   can be chosen to be locally path-connected.
 \end{defn}

 \begin{lem}{\label{L:stronglpc}}
   Let $\X$ be a locally path-connected
   topological stack with slice property.
   Assume that for every $x \in \X$ the inertia group $I_x$
   is locally path-connected. Then, $\X$ is strongly locally
   path-connected.
 \end{lem}

 \begin{proof}
   We may assume that $\X=[X/G]$, where $G$ is locally
   path-connected. Since $\X$ is locally path-connected, there is an
   atlas $p\: Y\to \X$ such that $Y$ is locally path-connected. Set
   $Z=Y\x_{\X}X$. Then, $Z$ is a $G$-torsor over $Y$. Since $Y$ and
   $G$ are both locally path-connected, so is $Z$. It follows from
   Lemma \ref{L:lpc} that $X$ is locally path-connected.
 \end{proof}

\begin{prop}{\label{P:stronglylc}}
   Let $\X$ be a locally path-connected topological stack.
   Assume either of the following holds:
  \begin{itemize}
    \item[$\mathbf{i.}$]  \, $\X$ is locally isomorphic to a quotient
     stack $[X/G]$ with $G$  a  Lie group and $X$ a Cartan $G$-space;

    \item[$\mathbf{ii.}$] \, $\X$ is the quotient stack of a proper
     Lie groupoid.
   \end{itemize}
 Then, $\X$ is strongly locally  path-connected.
\end{prop}

\begin{proof}
  Use Propositions \ref{P:strong1}, \ref{P:strong2},
  and Lemma \ref{L:stronglpc}.
\end{proof}


Recall from  Example 2 of $\S$\ref{SS:slice} that the Cartan
condition is automatically satisfied if any one of the following is
true: 1) $G$ is finite; 2)  $G$ is compact Lie and $X$ is completely
regular; 3)  $G$ is an arbitrary Lie group, $X$ is locally compact,
and the action is proper, i.e., $G\times X \to X\times X$
is a proper map.

  \begin{prop}{\label{P:strongcovering}}
    Let $p \: \Y \to \X$ be a covering map of topological stacks,
    and assume that $\X$ is  strongly locally path-connected. Then,
    for every point $x \in \X$, there exists an open substack $x\in
    \U \subseteq \X$ with the following properties:
    \begin{itemize}
      \item[$\mathbf{i.}$] $\U\cong[U/I_x]$, where $U$ is a locally
       path-connected topological space with an action of $I_x$ that
       fixes the (unique) lift of $x$ to $U$ (which we denote again
       by $x$) and has slice property at $x$;

      \item[$\mathbf{ii.}$] $p^{-1}(\U)\cong \coprod_{k\in K}
       [U/H_k]$, where $H_k$, for $k$ ranging in some index set $K$,
       are open-closed subgroups of $I_x$ acting on $U$ via $I_x$.
    \end{itemize}
        In particular,  $\Y$ is also strongly locally path-connected.
 \end{prop}

 \begin{proof}
    Denote $I_x$ by $G$ throughout the proof.

    By shrinking $\X$ around $x$ we may assume that $\X=[X/G]$,
    where $G$ and $X$ satisfy ($\mathbf{i}$). Consider the
    corresponding  atlas $X \to \X$, and let $Y \to \Y$ be the
    pull-back atlas for $\Y$. The map $q \: Y \to X$, being a pull
    back of $p$, is again a covering map. There is an open set $U
    \subseteq X$ containing $x$ over which $q$ trivializes. After
    replacing $U$ with a smaller open set containing $x$ (say, by
    the connected component of $x$ in $U$), we may assume that $U$
    is $G$-invariant and path-connected. Set $\U=[U/G]$. We claim
    that $\U$ has the desired property.

    Let $\V=p^{-1}(\U)\subseteq \Y$, and  $V=q^{-1}(U)\subseteq Y$.
    Then,  $V$ is an atlas for $\V$ and  is of the form
                  $$V=\coprod_{j\in J} U_j, \ U_j=U,$$
    for some    index set $J$. Set $R=V\x_{\V}V$ (so
    $\V\cong[V/R]$). Note that $V$ can  be viewed as an atlas for
    $\U$ too, and if we set $R'=V\x_{\U}V$, then the groupoid
    $[R\sst{}V]$ is an open-closed subgroupoid of $[R'\sst{}V]$
    (Corollary \ref{C:openclosed}). Observe that, as a topological
    space, $R'$ is homeomorphic to a disjoint union of $J\times J$
    copies of $G\x U$, and the restriction of $s \: R' \to V$ and $t
    \: R' \to V$ to each of these copies factors through some $U_j
    \subset V$ via a map that is isomorphic to the projection $G\x U
    \to U$. In particular, $s \: R' \to V$ (and also $t$) has the
    following two properties: 1) For every $W \subseteq R'$, $s(W)$
    is open in $V$ and the restriction $s|_{W}\: W \to s(W)$ admits
    local sections; 2) If $W$ is also
    closed, then $s(W)$ is a disjoint union of copies of $U$ in
    $V=\coprod U_j$.

    Now, observe that, since $[R\sst{}V]$ is an open-closed
    subgroupoid of $[R'\sst{}V]$, properties $(1)$ and $(2)$
    mentioned above also hold for $s,t \: R \to V$. An immediate
    consequence is that, for each $U_j \subseteq V=\coprod U_j$, the
    orbit $\OO(U_j)$ is a disjoint union of copies of $U$ in $V$; in
    particular $\OO(U_j) \subseteq V$ is open-closed. Therefore, by
    Lemma \ref{L:opensubstack}, $[U_j/R|_{U_j}]$ is an open-closed
    substack of $\V$.

    Let us analyze what  $[U_j/R|_{U_j}]$ looks like. Recall that
    $R|_{U_j}=(U_j\x U_j)\x_{V\x V}R$. Equivalently,
    $R|_{U_j}=s^{-1}(U_j)\cap t^{-1}(U_j)$. Note that
    $s^{-1}(U_j)\cong G\x U$. Hence, since $U$ is connected,
    $s^{-1}(U_j)\cap t^{-1}(U_j)$, being an open-closed subspace of
    $s^{-1}(U_j)\cong G\x U$, is of the form $H\x U$, where
    $H\subseteq G$ is an open-closed subspace. It is easy to see
    that $H$ is in fact a subgroup of $G$, so that the groupoid
    $[U_j/R|_{U_j}]$ is  simply (isomorphic to)
    the action groupoid of $H$ acting on $U$ via $G$.

    So, we have shown that $\V$ is a disjoint union of open-closed
    substacks of the form $[U_j/R|_{U_j}]$  each of which is
    equivalent to $[U/H]$ for some open-closed subgroup $H \subseteq
    G$. (Note that different $j$'s may correspond to the same
    substack $[U_j/R|_{U_j}]$.) The proof is complete.
 \end{proof}

\section{Relation between various Galois theories}{\label{S:Relation}}

In this section, we take a closer look at the Galois categories
introduced in $\S$\ref{SS:Galois1}-\ref{SS:Galois4} and study the
relations between them. The most important class for us is
$\mathsf{Free}$. In the next section, we interpret these results in
terms of fundamental groups of stacks and coarse moduli spaces.

\subsection{Relation between $\mathsf{Free}_{\X}$ and
$\mathsf{Full}_{\Xm}$}{\label{SS:FreeFull}}

  \begin{lem}{\label{L:invariance}}
     Let $\X$ be a topological stack, and let $\pim \: \X \to \Xm$
     be its moduli map. Let $Y$ be a topological space. Let $Y \to
     \Xm$ be a local homeomorphism, and set $\Y:=Y\x_{\Xm}\Xm$.
     Then, $Y$ is the coarse moduli space of $\Y$. That is, the
     induced map $g \: \Ym \to Y$ is a  homeomorphism.
   \end{lem}

 \begin{proof}
   By the very definition of the coarse moduli space, the statement
   is local on $Y$. That is, it is enough to prove the statement
   after replacing $Y$ by an open covering. So, we may assume that
   $Y \to \Xm$ is a disjoint union of open embedding, in which case
   the slemma is obvious.
 \end{proof}

 \begin{lem}{\label{L:free1}}
   Let $\X$ be a topological stack, and let $p \: \X \to A$ be an
   arbitrary map to a topological space $A$. Let $g \: B \to A$ be a
   covering map of topological spaces. Then, the induced map $f \:
   \X \x_A B \to \X$ is free ($\S$\ref{SS:Galois4}).
 \end{lem}

 \begin{proof}
   Denote $\X \x_A B$ by $\Y$. We have to show that $f$ is FPR and
   that the induced map $f_{mod} \: \Ym \to \Xm$ is a covering map.

   \vspace{0.1in}
    \noindent {\em Proof of FPR}.
     Choose an atlas $p \: X \to \X$, and let $q \: Y \to \Y$ be the
     pull-back atlas for $\Y$. Denote the corresponding groupoids by
     $[R_X\sst{}X]$ and $[R_Y\sst{}Y]$. Let $S_X \to X$ and $S_Y \to
     Y$ be the relative stabilizer  groups of these groupoids (that
     is, $S_X=X\x_{X\x_A X}R$, where $X\to X\x_A X$ is the
     diagonal). Observe that, for every point $x \in X$, the fiber
     $S_x$ of $R \to X\x_A X$ over the point $(x,x)$ is naturally
     isomorphic to the inertia group $I_{p(x)}$. The result now
     follows from the standard fact that the  diagram
         $$\xymatrix@=12pt@M=10pt{
           S_Y \ar[r]\ar[d]  & S_X  \ar[d]  \\
             Y \ar[r]        &   X    }$$
     is cartesian.

   \vspace{0.1in}
    \noindent {\em Proof that $f_{mod}$ is a covering map}.
     The map $p$ factors through the moduli map $\pim \: \X \to
     \Xm$. Set $Y':=B\x_A\Xm$. Then $Y' \to \Xm$ is a covering map.
     Using the fact that the diagram
       $$\xymatrix@=12pt@M=10pt{
          \Y \ar[r]\ar[d]  &  \X \ar[d]^{\pim}  \\
                        Y' \ar[r]        &  \Xm     }$$
     is cartesian, together with the fact that taking coarse moduli
     space commutes with base extension along covering maps (Lemma
     \ref{L:invariance}), we see  that $Y'$ is naturally homeomorphic
     to $\Ym$, and the map $Y' \to \Xm$ is naturally identified with
     $f_{mod}$.
 \end{proof}

 \begin{lem}{\label{L:free2}}
   Let $f \: \Y \to \X$ be a free covering stack. Then, the
   following diagram is cartesian:
    $$\xymatrix@=12pt@M=10pt{
           \Y \ar[r]^f \ar[d]  &  \X \ar[d]  \\
                 \Ym   \ar[r]_{f_{mod}}  &   \Xm  }$$
 \end{lem}

 \begin{proof}
    Set $\Y':=\Ym\x_{\Xm}\X$. Then $\Y' \to \X$ is a covering stack
    of $\X$, and there is a natural map $\Y \to \Y'$ of covering
    stacks over $\X$. This maps induces a bijection on the
    fibers, so it is an isomorphism.
 \end{proof}

Using the above three lemmas, the following proposition is immediate.

 \begin{prop}{\label{P:fundgpmoduli}}
  Let $\X$ be a topological stack. Then, there is an equivalence of
  categories
      $$\xymatrix@C=60pt{ \mathsf{Free}_{\X}
                     \ar @/^/[r]^{\txt\tiny{\bf{coarse moduli}}}  &
           \mathsf{Full}_{\Xm}
                  \ar @/^/ [l]^{\txt\tiny{\bf{base extension}\\
                                             {\bf via} $\pim$}}}.$$
  The similar statement is true for  connected covering stacks.
  Finally, the statement remains valid if we add the adjective
  `pointed'.
 \end{prop}

\subsection{Description of $\mathsf{FPR}$}{\label{SS:FPR}}

The next proposition leads to a satisfactory description of the
Galois category $\mathsf{FPR}$. Recall from (\cite{homotopy},
$\S$17) that, for every $x \in \X$, there is a natural group
homomorphism $\ox \: I_x \to \pX$.

 \begin{lem}{\label{L:FPRcovers}}
    Let $(\X,x)$ be a pointed connected topological stack, and let
    $f \: (\Y,y) \to (\X,x)$ be a pointed connected covering map.
    Then, we have the following:
  \begin{itemize}

    \item[$\mathbf{i.}$] $f$ is FPR at $y$ if and only if the
     corresponding subgroup $H \subseteq \pX$ contains $\ox(I_x)$;

    \item[$\mathbf{ii.}$] For every point $x'$, and every path
     $\gamma$ in $\X$ from $x'$ to $x$, identify $\omega_{x'}(I_{x'})
     \subseteq \pi_1(\X,x')$ with a subgroup of $\pX$ via the
     isomorphism $\gamma_* \: \pi_1(\X,x') \to \pX$. Let $N$ be the
     (necessarily normal) subgroup of $\pX$  generated by all these
     groups. Then, $f$ is FPR if and only if the corresponding
     subgroup $H \subseteq \pX$ contains $N$.
  \end{itemize}
 \end{lem}

 \begin{proof}[Proof of Part $(\mathbf{i})$]
  By Lemma 18.16 of \cite{homotopy}, we have a cartesian diagram
      $$\xymatrix@=12pt@M=10pt{
        I_y \ar[r]^(0.38){\oy} \ar[d]_{f_*}  &  \pY \ar[d]^{f_*}  \\
          I_x   \ar[r] _(0.38){\ox}       &   \pX    }$$
  So, the map $f_* \: I_y \to I_x$ (which is already injective
  because $f$ is representable) is an isomorphism if and only if the
  image of $f_* \: \pY \to \pX$, namely $H$, contains $I_x$.

 \vspace{0.1in}
  \noindent{\em Proof of Part}  ($\mathbf{ii}$). Lift $\gamma$ to a
  path in  $\Y$ ending at $y$, and call the starting point $y'$. Then
  $f$ is FPR at $y'$ if and only if $H$ contains
  $\gamma_*(\omega_{x'}(I_{x'}))$.
 \end{proof}

Let us rephrase part ($\mathbf{ii}$) of the above lemma as a proposition.

 \begin{prop}{\label{P:FPRcovers}}
  The open subgroups of the $\mathsf{FPR}$ topology on $\pi_1(\X,x)$ are
  precisely the open subgroups of $\pi_1(\X,x)$ in the full topology
  which contain $N$.
 \end{prop}

\subsection{Relation between $\mathsf{Free}$ and  $\mathsf{FPR}$}{\label{SS:FreeFPR}}

  The subcategory $\mathsf{Free} \subseteq \mathsf{FPR}$ is not as easy
  to describe in general. But we have the following results.

  \begin{prop}{\label{P:strongFPR}}
     If $\X$ is a strongly locally path-connected topological stack,
     then every FPR covering stack is automatically free. That is,
     $\mathsf{Free}=\mathsf{FPR}$.
  \end{prop}

  \begin{proof}
     This follows immediately from Proposition \ref{P:strongcovering}.
   \end{proof}

When $\X$ is not strongly locally path-connected, we could still say
something.

 \begin{prop}{\label{P:properFPR}}
     Let $[R\sst{}X]$ be a topological groupoid, and let
     $\X=[X/R]$ be its quotient stack. Assume either of the following
     holds:
      \begin{itemize}
          \item[$\mathbf{i.}$] $X$ is metrizable and  $s,t\: R\to X$
            are closed maps;

          \item[$\mathbf{ii.}$] $X$ is Hausdorff and  $s,t\: R\to X$ are
             proper.
      \end{itemize}
     Then, every finite covering stack of $\X$ that is FPR is
     automatically free. That is,
     $\mathsf{FPR}\cap \mathsf{Fin}  =\mathsf{Free}\cap \mathsf{Fin}$.
 \end{prop}

 \begin{proof}
  We may assume that $\X$ is connected. Let $f\: \Y \to \X$ be a
  connected finite FPR covering stack of degree $n$. We have to show
  that $f_{mod} \: \Ym \to \Xm$ is also a covering map.

  Consider the atlas $p\: X \to \X$, and let $q \:Y \to \Y$ be the
  pull-back atlas for $\Y$. Let $[R_Y\sst{}Y]$ be the corresponding
  groupoid. To have consistent notation, we denote $R$ by $R_X$. The
  map $g\: Y \to X$ is a covering map of degree $n$. Recall
  (\cite{homotopy}, Example 4.13) that  $\Xm$ is homeomorphic to the
  coarse quotient of $X$ by the equivalence relation induced from
  $R_X$ (and similarly for $\Ym$). We have a commutative diagram
     $$\xymatrix@=12pt@M=10pt{
            Y \ar[r]^(0.4)h \ar[d]_g &  \ar[d]^{f_{mod}} \Ym \\
                              X  \ar[r]        &   \Xm        }$$
  By Lemma \ref{L:FPR}, both vertical maps have constant degree $n$.
  Therefore, $h$ is a fiberwise bijection. Take a point $\bar{x}
  \in \Xm$, and pick a lift $x \in X$ for it. Let
  $\bar{y}_1,\cdots,\bar{y}_n \in \Ym$ be the elements of the fiber of
  $f_{mod}$ over $\bar{x}$. Similarly, let $y_1,\cdots,y_n \in Y$ be
  the elements of the fiber of $g$ over $x$. Consider the orbit
  $B_i:=\OO(y_i) \subseteq Y$ of $y_i$ under the action of the
  groupoid $[R_Y\sst{}Y]$ (this is simply the fiber of $h \: Y \to
  \Ym$ over $y_i$). By hypothesis, $B_i$ is  closed. Since $h$ is a
  fiberwise bijection, the restriction $g|_{B_i} \: B_i \to X$ is
  injective for every $i$. It is also closed, because $g$ is a finite
  cover. Therefore, $g|_{B_i} \: B_i \to A$ is a homeomorphism for
  every $i$, where $A=g(B_1)=\cdots=g(B_n)=\OO(x)$. In other words,
  the covering map $g \: Y \to X$ trivializes over $A \subseteq X$. We
  claim that there exists an open $A\subseteq U$ such that $g$
  trivializes over $U$ as well. By condition ($\mathbf{i}$) or
  ($\mathbf{ii}$), we can find open sets $B_i \subseteq V_i$ such that
  $V_i\cap V_j=\emptyset$, for every $i\neq j$   (see Lemma
  \ref{L:metrizable} below). By shrinking each $V_i$, we may assume
  that $g|_{V_i} \: V_i \to U$ is a homeomorphism for every $i$. It is
  easy to check that $U:=\cap g(V_i)$ has the desired property.

  The next claim is that, after some more shrinking, we may assume
  that $U$ and $V_i$, $i=1,\cdots,n$, are invariant open sets for
  the corresponding groupoids. To do so, note that the source and
  target maps of the groupoids $[R_X\sst{}X]$ and $[R_Y\sst{}Y]$ are
  both closed maps. This follows from the hypothesis and the fact
  that the diagram
    $$\xymatrix@=12pt@M=10pt{
          R_Y \ar[r]^s \ar[d] &  Y\ar[d]^g  \\
           R_X   \ar[r]_s        &   X        }$$
  is cartesian. So, we may replace $U$ by
  $U-s\big(R_X-t^{-1}(U)\big)$, and similarly, each $V_i$ by
  $V_i-s\big(R_Y-t^{-1}(V_i)\big)$. Note that we still have
  $A\subseteq U$, $B_i \subseteq V_i$, and each $V_i$ maps
  homeomorphically to $U$. Hence, after passing to the coarse moduli
  spaces, we obtain an  open neighborhood of $\bar{U}$ of $\bar{x}$
  over which $f_{mod}$ trivializes as an $n$-sheeted covering (the
  sheets being $\bar{V}_1,\cdots,\bar{V}_n$).  The proof is complete.
 \end{proof}

  \begin{lem}{\label{L:metrizable}}
    Let $f \: Y \to X$ be a finite covering map of topological
    spaces. If $X$ is metrizable, then so is $Y$
  \end{lem}

   \begin{proof}
    This follows from Smirnov Metrization Theorem which says that a
    topological space $X$ is metrizable if and only if it is
    paracompact, Hausdorff, and locally metrizable. All these
    properties are easily seen to be stable
     under passing to finite covering spaces.
   \end{proof}

\subsection{Description of $\mathsf{Full}$ and $\mathsf{Fin}$ when
$\X$ is semilocally 1-connected}{\label{SS:semilocally}}

When $\X$ is semilocally 1-connected things are as nice as they can
be, because the (pointed) covering stacks of $\X$ are in a bijection
with subgroups of $\pi_1(\X,x)$; see \cite{homotopy}, $\S$18.2.

 \begin{prop}{\label{P:semilocally}}
  Suppose that $\X$ is a connected, locally path-connected, semilocally
  1-connected topological stack. Then, $\mathsf{Full}$ corresponds to
  the discrete topology on $\pi_1(\X,x)$ and $\mathsf{Fin}$
  corresponds to the profinite topology. We have $\pi'_1(\X,x)=\pX$,
  and $\pi_1(\X,x)^{\wedge}_{\mathsf{Fin}}=\widehat{\pi_1(\X,x)}$,
  the profinite completion of $\pi_1(\X,x)$.
 \end{prop}

 \begin{proof}
  The statement about $\mathsf{Full}$ follows from Proposition
  \ref{P:discrete}. The second statement is obvious.
 \end{proof}

\section{Fundamental group of the coarse moduli space}{\label{S:Moduli}}

In this section, we translate the results of the previous section in
terms of the fundamental groups. The outcome is some formulas for
the fundamental group of the coarse moduli space of a topological
stack.

\vspace{0.1in}

\noindent {\bf Notation.} Throughout this section, the group $N
\subseteq \pX$ refers to the group defined in Lemma
\ref{L:FPRcovers}.

 \begin{prop}{\label{P:fundgpcoarsemod}}
    Let $\X$ be a connected locally path-connected topological stack.
    Then, we have the following:
    \begin{itemize}
      \item[$\mathbf{i.}$] The image of the natural map $\pim'_{*}\:
       \pi'_1(\X,x) \to \pi'_1(\Xm,x)$ is not contained in any proper
       open subgroup of $\pi'_1(\Xm,x)$. In particular, if the full
       topology on $\pXm$  has a basis of open neighborhoods
       (equivalently, every covering space of $\Xm$ can be dominated
       by a Galois covering), then $\pim'_{*}$ has a dense image.

      \item[$\mathbf{ii.}$] There is a natural isomorphism
         $$\pi_1(\X,x)^{\wedge}_{\mathsf{Free}} \risom \pi'_1(\Xm,x).$$

      \item[$\mathbf{iii.}$] If $\X$ is  strongly locally
       path-connected, then we have a natural
       isomorphism $$(\pi_1(\X,x)/N)' \risom \pi'_1(\Xm,x).$$ (The
       left hand side is the completion with respect to the quotient
       topology induced on $\pi_1(\X,x)/N$ from the full topology of
       $\pi_1(\X,x)$.)

      \item[$\mathbf{iv.}$] If $\X$ is as in Proposition
       \ref{P:properFPR}, then we have a natural isomorphism
         $$(\pi_1(\X,x)/N)^{\wedge}_{\mathsf{Fin}}
                      \risom{\pi_1(\Xm,x)}^{\wedge}_{\mathsf{Fin}}.$$
   \end{itemize}
 \end{prop}

 \begin{proof}[Proof of Part  $(\mathbf{i})$]
   This is equivalent to saying that the pull-back via $\pim \: \X
   \to \Xm$ of a connected covering space  $Y$ of $\Xm$  remains
   connected. Denote this pull-back by $\Y$. Since $\X$ is locally
   connected, so is $\Y$. If $\Y$ is not connected, we can write it
   as a disjoint union of two open-closed substacks $\Y_1 \coprod
   \Y_2$. But then, by Lemma \ref{L:invariance}, we would have a
   decomposition $\Y_{1,mod} \coprod  \Y_{2,mod}$ of $Y$ into
   open-closed subspaces, which is impossible.

  \vspace{0.1in}

  \noindent{\em Proof of Part}  ($\mathbf{ii}$).
   This follows from Proposition \ref{P:fundgpmoduli}.

  \vspace{0.1in}

  \noindent{\em Proof of Part}  ($\mathbf{iii}$).
   This follows from Part ($\mathbf{ii}$) and Proposition
   \ref{P:FPRcovers} (also see
   Example \ref{E:prodiscrete2}.$\mathbf{1}$).

  \vspace{0.1in}

  \noindent{\em Proof of Part}  ($\mathbf{iv}$).
   This follows from Proposition \ref{P:properFPR}.
 \end{proof}

In general, it is desirable to work with the actual fundamental
group $\pX$ (i.e., the one defined using loops) rather than the
fancy prodiscrete fundamental groups of $\S$\ref{SS:Galois1}. But to
do so one needs to assume that $\X$ is semilocally 1-connected. In
what follows we analyze what happens in the presence of this
condition. Recall from $\S$\ref{SS:semilocally} that in this case
the full topology on $\pX$ is discrete.

 \begin{cor}{\label{C:surjective}}
  Let $\X$ be a connected locally path-connected topological stack.
  Assume $\Xm$ is semilocally 1-connected. Then, the map $\pim_{*}
  \: \pX \to \pXm$ is surjective.
 \end{cor}

 \begin{proof}
2     Let $H \subseteq \pXm$ be the image of $\pim_{*} \: \pX \to
     \pXm$. Since  $\pX \to H$ is  continuous, there is a unique
     extension $\pim'_1(\X,x) \to H$, by  Corollary \ref{C:extend}.
     This extension coincides with $\pim'_{*}$ of Proposition
     \ref{P:fundgpcoarsemod}.$\mathbf{i}$. Therefore, $H$ must be
     dense in $\pXm$. But $\pXm$ is discrete by
     Proposition \ref{P:semilocally}. So $H=\pXm$.
 \end{proof}

 \begin{thm}{\label{T:semilocally}}
    Let $(\X,x)$ be a connected topological stack. Assume
    $\X$ and $\Xm$ are semilocally 1-connected.

     \begin{itemize}
       \item[$\mathbf{i.}$] If  $\X$ is strongly locally
        path-connected,  then we have a natural isomorphism
                $$\pi_1(\X,x)/N \risom \pi_1(\Xm,x).$$

       \item[$\mathbf{ii.}$] Assume $\X$ is the quotient stack of a
         topological groupoid $[R\sst{}X]$ in which either $X$ is
         metrizable and the source map $s \: R \to X$ is closed, or
         $X$ is Hausdorff and the source map $s \: R \to X$ is
         proper. Then, we have a natural isomorphism
                $$\widehat{\pi_1(\X,x)/N}
                      \risom\widehat{\pi_1(\Xm,x)},$$
         where \ $\hat{}$\ stands for profinite completion.  (The
         \ $\hat{}$\ on the left hand side is over the whole expression.)
     \end{itemize}
 \end{thm}

  \begin{proof}
    Use Proposition \ref{P:fundgpcoarsemod}.$\mathbf{iii}$,$\mathbf{iv}$
    and Proposition \ref{P:semilocally}.
  \end{proof}

\begin{rem}{\label{R:semilocally}}
 It is easy to see (\cite{homotopy}, Lemma 18.4) that, if $\X$ is a
 topological stack which is (locally) path-connected, then so is
 $\Xm$. The similar statement is not true for semilocally
 1-connected stacks in general. But it is true when $\X=[X/G]$, where
 $G$ is a Lie group and $X$ is a Cartan $G$-space. This follows from
 \cite{Palais}, Theorem 2.3.3. In the case where $G$ is a compact
 Lie group this is Corollary 6.4 in Chapter II of \cite{Bredon}.
\end{rem}

 \begin{ex}{\label{E:graphs}}
  Let $\mathcal{G}$ be a graph-of-groups. It is shown in
  \cite{homotopy} that $\mathcal{G}$ can be realized as a
  Deligne-Mumford topological stack. For each vertex $v$ of
  $\mathcal{G}$, let $G_v$ denote the corresponding group. Then, we
  have $G_v=I_v$, and the homomorphism $G_v \to
  \pi_1(\mathcal{G},v)$ defined by Serre coincides with our map
  $\omega_v \: I_v \to \pi_1(\mathcal{G},v)$. The similar thing is
  true for the homomorphisms $G_e \to \pi_1(\mathcal{G},e_0)$, where
  $e_0$ is any point on the edge $e$.

  If we let $N$ be the normal subgroup generated by the images of
  all $G_v$ and $G_e$ in $\pi_1(\mathcal{G})$, we find that the
  fundamental group of the underlying graph of $\mathcal{G}$ is
  isomorphic to $\pi_1(\mathcal{G})/N$. This is of course a
  well-known formula (see \cite{Bass}, Example 2.14). The similar
  result is valid for complexes-of-groups as well, and this
  should presumably be well-known too.
 \end{ex}

 \begin{ex}{\label{E:Hawaiian2}}
  Let $X$ and $Y$ be as in Example \ref{E:Hawaiian}. There is a
  shift action of $\bbZ$ on $X$ and $Y$. Set $\X=[X/\bbZ]$ and
  $\Y=[Y/\bbZ]$. The coarse moduli space of $\Y$ is $Y/\bbZ$, which
  is $S^1$ (with the usual topology). The coarse moduli space of
  $\X$ is $S^{1}$ with a topology that is the usual topology away
  from the base point $\bullet \in S^1$, but the only open
  containing $\bullet $ is the entire space;
  this is a contractible space.

  The fundamental group of $Y$ is $\mathbb{F}^{\bbZ}=<a_i \ |\ i\in
  \bbZ>$. The fundamental group of $X$ contains $\mathbb{F}^{\bbZ}$,
  but it is considerably bigger (it is indeed uncountable). The
  $\bbZ$-actions on $X$ and $Y$ induce  $\bbZ$-actions on $\pi_1(X)$
  and $\pi_1(Y)$. In the case of $Y$, this action is simply given by
  shifting the generators. In the case of $X$, it is also some sort
  of a shift, but it is more complicated to
  describe. We have
     $$\pi_1(\Y)\cong \mathbb{F}^{\bbZ}\rtimes \bbZ, \ \text{and} \
              \pi_1(\X)\cong \pi_1(X)\rtimes \bbZ.$$

  The image of the map $\omega_{\bullet} \: \bbZ=I_{\bullet} \to
  \pi_1(\Y,\bullet)$ is simply the factor $\bbZ$ of the semi-direct
  product (similarly for $\pi_1(\X)$). If we kill this subgroup of
  $\pi_1(\Y)$ we get a group that is isomorphic to
  $\bbZ\cong\pi_1(\Xm)$, as predicted by Theorem
  \ref{T:semilocally}. (In fact, it seems that the same is true for
  $\Y$. That is, if we kill $\bbZ$ in $\pi_1(\X)$ we end up with the
  trivial group. Note that this is not predicted by Theorem
  \ref{T:semilocally}, as $\Y$ is not semilocally 1-connected.)
 \end{ex}

\section{Application to quotients of group actions; the exact sequence}
{\label{S:Quotient1}}

Results of the previous section can be used to compute the
fundamental group of the coarse quotient space of a topological
group action. The idea is that the fundamental group of the quotient
{\em stack} $[X/G]$ of a group action is easy to compute thanks to
the fiber homotopy exact sequence of the fibration $G \to X \to
[X/G]$. Theorem \ref{T:semilocally} then can be used to compute the
fundamental group of $[X/G]_{mod}=X/G$ from the fixed point data of
the action (which manifests itself in the inertia groups of $[X/G]$).

\vspace{0.1in}

\noindent{\bf Conventions.} In this and the next sections, $G$ will
be a topological group acting continuously on a connected, locally
path-connected, semilocally 1-connected topological space $X$. We
assume that the coarse quotient $X/G$ is also semilocally
1-connected; see Remark \ref{R:semilocally}. We denote the group of
path components of $G$ by $\mathbf{G}$. We let $\mathbf{I} \subset
\mathbf{G}$ be the set of all path components of $G$ which contain
an element $g$ such that the action of $g$ on $X$ has a fixed point.

 \begin{thm}{\label{T:fundgpquotient1}}
  Let $G$ and $X$ be as above. Suppose that the action has slice
  property.  Assume that all stabilizer groups $I_x$ of the action are
  locally path-connected. Fix a base point $x$ in $X$ and let
  $\bar{x}$ be its image in $X/G$.  Then, we have an exact sequence
         $$\pi_1(X,x) \to \pi_1(X/G,\bar{x}) \to \mathbf{G/I} \to 1.$$
  In particular, if $X$ is simply-connected, then we have an
  isomorphism
         $$ \pi_1(X/G) \risom \mathbf{G/I}.$$
 \end{thm}

 \begin{proof}
  Set $\X=[X/G]$. We will abuse the notation and  denote the images
  of $x$ in $[X/G]$ also by $\bar{x}$.

  By Lemma \ref{L:stronglpc}, $\X$ is a strongly locally
  path-connected topological stack. Hence, we can apply Theorem
  \ref{T:semilocally}.$\mathbf{i}$. To get the desired exact
  sequence, we computed the group $N$
  appearing in Theorem \ref{T:semilocally}.

  Since $p$ is a Serre fibration with fiber $G$, we have a fiber
  homotopy exact sequence $$\pi_1(X,x) \to \pi_1(\X,\bar{x}) \to
  \mathbf{G} \to 1.$$ It is enough to show that
  the image of $N$ in $\mathbf{G}$  is equal to $\mathbf{I}$.

  We have a natural isomorphism $\varphi \: I_{\bar{x}} \risom
  I_{x}$, where $I_{x} \subseteq G$ is the stabilizer group of $x$.
  The claim is that the following diagram is commutative:
     $$\xymatrix@=12pt@M=10pt{
          I_{\bar{x}}  \ar[r]^(0.37){\omega_{\bar{x}}}
                \ar[d]_{\varphi}^{\cong} &
                                \pi_1(\X,\bar{x}) \ar[d]  \\
             I_{x} \ar@{^(->} [r]        & \mathbf{G}        }$$

  Take an element $\gamma \in I_{\bar{x}}$. Recall that the corresponding
  loop $\tilde{\gamma}:=\omega_{\bar{x}}(\gamma) \: S^1 \to \X$ is
  defined by the triple $(h,\ep_0,\ep_1)$, where $h \: [0,1]\to \X$
  is the constant path at $\bar{x}$, $\ep_0=\id$, and $\ep_1=\gamma$
  (see \cite{homotopy}, $\S$17 for definitions). We have to show that
  the image of $\tilde{\gamma}$ in $\mathbf{G}$ is equal to
  $\varphi(\gamma)$.

  Consider the pull-back $T:=S^1\x_{\X}X$ of $X$ over $S^1$, where
  the map $S^1 \to \X$ is the $\tilde{\gamma}$ defined above. Note that
  $T$ is naturally pointed with the point $*=(\bullet,\id_{\bar{x}},x)$
  sitting above $\bullet \in S^1$. Here, $\id_{\bar{x}}$ is the identity
  transformation from $\tilde{\gamma}(\bullet)=\bar{x}$ to $p(x)=\bar{x}$.

  As a $G$-torsor over $S^1$, $T$ can be described as the quotient of
  $[0,1]\x G$ by the relation $(1,g)\sim\big(0,g
  \varphi(\gamma)\big)$. Under this identification, $*$ corresponds
  to the point $(0,1_G)$, where $1_G$ is the identity element of $G$.

  To find the image of $\tilde{\gamma}$ in  $\mathbf{G}$, we have to
  choose a lift $F$ in  diagram
     $$\xymatrix@=18pt@M=6pt{
                     &  \ar[d] T & G \ar@{_(->} [l]  \ar[d] \\
              [0,1] \ar[r]   \ar @{..>} [ur]_F \ar@{} @<2ex> [ur]
                |{\rotatebox{45}{{\scriptsize$0\mapsto *$}}}
                            &    S^1  & \bullet \ar@{_(->}[l] }$$
  and determine the path  component of $F(1) \in G$. But from the
  description of the torsor $T$, it is clear that $F(1)$ is in the
  same path component as $\varphi(\gamma) \in G$. This
  proves the commutativity of the above square.

  We can now repeat the same argument for different choices of base
  points $x' \in X$, and then transport the situation over to $x$ by
  choosing a path connecting $x$ and $x'$. This shows that
  the image of $N$ in $\mathbf{G}$ is $\mathbf{I}$.
 \end{proof}

 \begin{rem}{\label{R:automatic}}
  All the hypotheses of Theorem \ref{T:fundgpquotient1} are
  automatically satisfied if $G$ is a Lie group and $X$ is a
  connected, locally path-connected, semilocally 1-connected Cartan
  $G$-space. See $\S$\ref{SS:slice} and
  Remark \ref{R:semilocally}
 \end{rem}

 \begin{ex}{\label{E:Brower}}
  Let $G$ be a compact Lie group acting on $S^{n}$. If $n$ is even,
  then $S^{n}/G$ is simply-connected (Brower's fixed point theorem).
  The same is true if $n$ is arbitrary and $G$ is connected.
 \end{ex}

 \begin{ex}{\label{E:Fuchsian}}
  Let $G \subset \PSL_2(\bbR)$ be a Fuchsian group. Assume that the
  fundamental domain of $G$ is not compact. Then, $G\cong T\times F$,
  where $T$ is torsion and $F$ is free.

  To prove this, note that the action of $G$ on the upper half-plane
  is properly discontinuous and the coarse quotient of this action
  is a non-compact Riemann surface. Let $T$ be the subgroup
  generated by all torsion elements of $G$.   By Theorem
  \ref{T:fundgpquotient1}, $G/T$ is isomorphic the fundamental group
  $F$ of this non-compact Riemann surface, hence it is free.
 \end{ex}

 \begin{ex}{\label{E:weighetd}}
  The weighted projective space $\mathbb{P}(n_0,\cdots n_k)$, $k\geq
  1$, is defined to be the quotients space
  $\bbC^{n+1}-\{0\}/\bbC^*$, where $\lambda \in \bbC^*$ acts on
  $\bbC^{n+1}-\{0\}$ via multiplication by
  $(\lambda^{n_0},\cdots,\lambda^{n_k})$. It is a standard fact that
  $\mathbb{P}(n_0,\cdots n_k)$ is simply-connected. This easily
  follows from Theorem \ref{T:fundgpquotient1}. More generally,
  whenever $G$ is a compact connected Lie group acting on a
  simply-connected completely regular topological space $X$, then
  $X/G$ is simply-connected. Indeed, it is true that the stack
  quotient $[X/G]$ is simply connected. This follows from the
  fiber homotopy  exact sequence for the fibration $X \to [X/G]$.
 \end{ex}

 \begin{thm}{\label{T:fundgpquotient2}}
  Assume that $G$ is a compact topological group acting on a
  Hausdorff space $X$. Then, we  have an exact sequence
         $$\widehat{\pi_1(X)} \to \widehat{\pi_1(X/G)}
                          \to \widehat{\mathbf{G/I}} \to \{*\},$$
  where \ $\hat{}$ \ stands for profinite completion.
  In particular, when $X$ is simply-connected we have an isomorphism
         $$\widehat{\pi_1(X/G)} \risom \widehat{\mathbf{G/I}}.$$
 \end{thm}

 \begin{proof}
  The proof of Theorem \ref{T:fundgpquotient1} is  valid up to the
  point that there is an exact sequence
        $$\pi_1(X) \to \pi_1(\X)/N \to \mathbf{G/I} \to 1$$
  However, we can not say that $\pi_1(\Xm)\cong \pi_1(\X)/N$. But we
  can say so after passing to profinite completion (Theorem
  \ref{T:semilocally}.$\mathbf{ii}$). The result
  now follows from the fact that profinite completion is right exact.
 \end{proof}


\section{Application to quotients of group actions; the explicit form}
{\label{S:Quotient2}}

In this section we give a more explicit version of Theorem
\ref{T:fundgpquotient1}. Without loss of generality, we will assume
that the action has a global fixed point. Remark \ref{R:trick}
explains why we can make this assumption.

\vspace{0.1in}
\noindent{\em Construction.}
Let $G$ be a group acting continuously on a topological space $X$.
Assume that there is $x\in X$ which  is fixed by the entire action.
For every triple $(g,y,\gamma)$, with $g \in G$, $y \in X^g$, and
$\gamma$  a path from $x$ to $y$, define $\lambda_{g,y,\gamma}\in
\pi_1(X,x)$ to be the loop $\gamma(g\gamma)^{-1}$. Let $K \subseteq
\pi_1(X,x)$ be the subgroup generated by all such
$\lambda_{g,y,\gamma}$. This is easily seen to
be a normal subgroup.
\vspace{0.1in}

We will abuse the notation and  denote the images of $x$ in $[X/G]$
and $X/G$ both by $\bar{x}$.

 \begin{thm}{\label{T:fixed}}
  Let $X$  and $G$ be as in Theorem \ref{T:fundgpquotient1} (also
  see Remark \ref{R:automatic}). Assume that there is $x\in X$ that
  is fixed by the action. Let $K$ be the subgroup of $\pi_1(X,x)$
  defined in the previous paragraph. Then, we have a natural
  isomorphism
       $$\pi_1(X,x)/K \risom \pi_1(X/G,\bar{x}).$$
 \end{thm}

 \begin{proof}
  Set $\X=[X/G]$, and let $p \: X \to \X$ denote the quotient map.

  First, we show that $\pi_1(\X,\bar{x})\cong
  \pi_1(X,x)\rtimes\mathbf{G}$. Here, $\mathbf{G}$ stands for the
  group of path components of $G$, and the action of $\mathbf{G}$ on
  $\pi_1(X,x)$ is the obvious one. The map $X \to \X$ is a Serre
  fibration with fiber $G$, so it gives rise
  to the fiber homotopy exact sequence
      $$\pi_1(G) \to \pi_1(X,x) \to \pi_1(\X,\bar{x})
                                   \to \mathbf{G} \to 1.$$
  It is easily seen that the leftmost map is the trivial map,
  so we actually have a short exact sequence
    $$1 \to \pi_1(X,x) \to \pi_1(\X,\bar{x}) \to \mathbf{G} \to 1.$$
  Since $x$ is the fixed point of the entire action, we have
  $I_{\bar{x}}=G$. The map $\omega_{\bar{x}} \: G=I_{\bar{x}} \to
  \pi_1(\X,\bar{x})$ factors through $\mathbf{G}$ to produce the
  desired splitting of the above short exact sequence.

  If $N \subseteq \pi_1(X,x)\rtimes\mathbf{G}$ is as in Theorem
  \ref{T:semilocally}.$\mathbf{i}$, we have
     $$\big(\pi_1(X,x)\rtimes\mathbf{G}\big)/N
            \cong\pi_1(\Xm,\bar{x})=\pi_1(X/G,\bar{x}).$$
  We will show that the left hand side is isomorphic to
  $\pi_1(X,x)/K$.

  Recall from Lemma \ref{L:FPRcovers} that, to define $N \subseteq
  \pi_1(\X,\bar{x})$, what we do is that we pick a point $x' \in \X$
  and identify $\omega_{x'}(I_{x'})$ with various subgroups of
  $\pi_1(\X,\bar{x})$ by taking  paths connecting $x'$ to $\bar{x}$.
  For any fixed $x' \in \X$,  denote  the subgroup generated by all
  these groups by $N_{x'} \subset \pi_1(\X,\bar{x})$. The group $N$
  is then  the one generated by all $N_{x'}$, $x' \in \X$. In fact,
  to generate $N$ it is enough to take all subgroups of the form
  $N_{p(y)}$, $y \in X$. Furthermore, to generate $N_{p(y)}$ it
  suffices to join $p(y)$ to $\bar{x}$ by paths of the from
  $p(\gamma)$, where  $\gamma$ is a path in $X$ joining $y$ to $x$.

  The group $N_{p(y)}\subseteq \pi_1(\X,x)=\pi_1(X,x)\rtimes\mathbf{G}$
  can  now be explicitly described by
     $$N_{p(y)}=\big\{\big(p(\lambda_{g,y,\gamma}),g\big)
          \ | \ g \in I_y, \ \gamma=\text{path in $X$ joining $x$
                                                      to $y$}\big\}.$$
  (This is not completely obvious. The proof requires a little bit
  of straightforward path chasing that we omit here.) So, to obtain
  $\big(\pi_1(X,x)\rtimes\mathbf{G}\big)/N$ we have to kill all
  elements of the form $\big(p(\lambda_{g,y,\gamma}),g\big)$ in
  $\pi_1(X,x)\rtimes\mathbf{G}$. Notice that this includes all
  elements of the form $(1,g)$, because $\lambda_{g,x,const.}=1$.
  Therefore, to obtain $\big(\pi_1(X,x)\rtimes\mathbf{G}\big)/N$ we
  have to kill all elements of the from $(1,g)$ and all elements of
  the form $\big(p(\lambda_{g,y,\gamma}),1\big)$ in
  $\pi_1(X,x)\rtimes\mathbf{G}$. The outcome of this
  is exactly $\pi_1(X,x)/K$. The proof is complete.
 \end{proof}

 \begin{rem}{\label{R:coinvariant}}
  If we kill all loops of the form $\lambda_{g,x,\gamma}$ in
  $\pi_1(X,x)$ we obtain $\pi_1(X,x)_G$, the group of coinvariants
  of the action of $G$ on $\pi_1(X,x)$. Therefore, there is a
  surjective homomorphism $\pi_1(X,x)_G \to \pi_1(X/G,x)$. This map
  is not necessarily an isomorphism. For instance, consider the
  action of $\bbZ/2\bbZ$ on $S^1$ defined by flipping along the
  $x$-axis. In this case, $\pi_1(X,x)_G=\bbZ/2\bbZ$,
  whereas $\pi_1(X/G,x)$ is trivial.

  One may wonder if the knowledge of the action of $G$ on
  $\pi_1(X,x)$ is enough to determine $\pi_1(X/G,\bar{x})$. The
  answer is no. For instance, in the above example we have
  $G=\bbZ/2\bbZ$, $\pi_1(X)=\bbZ$,  the action is $n \mapsto  -n$,
  and we have $\pi_1(X/G)=0$. Now, let $X$ be $S^2$ with the north
  and the south poles joined by a straight line. Take
  $G=\bbZ/2\bbZ$, and let the action be antipodal. In this case, we
  have, as in the previous example, $G=\bbZ/2\bbZ$, $\pi_1(X)=\bbZ$,
  and the action of $G$ is  $n  \mapsto  -n$. However, $X/G$ is
  homotopy equivalent to $\bbR P^2$, so $\pi_1(X/G)=\bbZ/2\bbZ$.

  Indeed, it is true in general that, if $G$ acts freely away from
  $x$, then $\pi_1(X,x)_G \to \pi_1(X/G,x)$ is an isomorphism. This
  follows easily from Theorem \ref{T:fixed}. The rest of this
  section is devoted to understanding the
  map $\pi_1(X,x)_G \to \pi_1(X/G,x)$ in general.
 \end{rem}

The subgroup $K$ appearing in Theorem \ref{T:fixed} may look a bit
too complicated to compute in general, because it seems to require a
detailed knowledge of the various fixed point sets $X^g$. In the
next theorem, we build on the idea discussed in Remark
\ref{R:coinvariant} and give a more efficient formula for
$\pi_1(X/G)$.

\vspace{0.1in}
\noindent{\em Construction.}
For every $g \in G$, {\em choose} a set of paths
$\{\gamma_{g,i}\}_{i \in \pi_0X^g}$  with the property that
$\gamma_{g,i}$ starts from $x$ and ends on the path component of
$X^g$ corresponding to $i$. Let $\pi_1(X,x)_G$ be the group of
coinvariants of the action of $G$ on $\pi_1(X,x)$, that is, the
largest quotient of $\pi_1(X,x)$ on which $G$ acts trivially. Denote
the image of $\gamma_{g,i}(g\gamma_{g,i})^{-1}$ in $\pi_1(X,x)_G$ by
$\bar{\lambda}_{g,i}$. Let $\bar{K} \subseteq \pi_1(X,x)_G$ be the
normal subgroup generated by $\{\bar{\lambda}_{g,i}\}_{g\in G,i\in
\pi_0X^g}$ (that is, the smallest normal subgroup of $\pi_1(X,x)_G$
that contains all $\bar{\lambda}_{g,i}$).

 \begin{rem}{\label{R:center2}}
  The generating loops we used to define the group $\bar{K}$ are
  superfluous. Indeed, we could define $\bar{K}$ to be the normal
  subgroup generated by $\{\bar{\lambda}_{g,i}\}$, where $g$ runs in
  $G$ and $i$ runs in $\pi_0 X^g/C(g)$. Here, $C(g)$ stands for the
  centralizer of $g$ in $G$.
 \end{rem}

 \begin{thm}{\label{T:fixed2}}
  Let $X$  and $G$ be as in Theorem \ref{T:fundgpquotient1} (also
  see Remark \ref{R:automatic}). Assume that there is $x\in X$ that
  is fixed by the action. Let $\bar{K}$ be the subgroup of
  $\pi_1(X,x)_G$ defined above. Then, we have a natural isomorphism
     $$\pi_1(X,x)_G/\bar{K} \risom \pi_1(X/G,\bar{x}).$$
 \end{thm}

 \begin{proof}
  We will show that $\bar{K}$ is the image of $K$ in $\pi_1(X,x)_G$,
  where $K$ is as in Theorem \ref{T:fixed}.
  The first observation is that
  $\lambda_{g,y,\gamma}=\lambda_{g,y',\gamma'}$, if $\gamma$ and
  $\gamma'$ are homotopic via a homotopy that is relative to both
  $\{x\}$ and  $X^g$. Therefore, to define $K$ it is enough to
  choose  a set of representative elements
  $\{y_{g,i}\}_{i\in\pi_0X^g}$ for path components of $X^g$ and only
  use paths $\lambda_{g,\gamma,y}$
  whose end points $y$  belong to this representative set.

  We now fix $g$ and $y$ and analyze what happens if we replace the
  path $\gamma$ with another path $\gamma'$ joining $x$ to $y$. Set
  $\alpha:=\gamma'\gamma^{-1} \in \pi_1(X,x)$, and
  $\beta:=\alpha(g\alpha)^{-1}$. It is easy to check that we have
  the equality
    $$\lambda_{g,y,\gamma'}=\alpha\lambda_{g,y,\gamma}\alpha^{-1}\beta$$
  in $\pi_1(X,x)$.  Since the image of $\beta$ in $\pi_1(X,x)_G$ is
  trivial, the images $\bar{\lambda}_{g,y,\gamma}$ and
  $\bar{\lambda}_{g,y,\gamma'}$ of these loops are conjugate in
  $\pi_1(X,x)_G$. Thus, a normal subgroup that contains one will
  necessarily contain the other. Therefore, as far as generating a
  {\em normal} subgroup is concerned, we can choose either of
  $\bar{\lambda}_{g,y,\gamma}$ and $\bar{\lambda}_{g,y,\gamma'}$.
  That is, the normal subgroup of $\pi_1(X,x)_G$ generated by the
  elements $\bar{\lambda}_{g,i}$ used in the construction
  of $\bar{K}$ is equal to the normal subgroup generated by
  all $\bar{\lambda}_{g,y,\gamma}$. In other words, $\bar{K}$
  is exactly the image of $K$ in $\pi_1(X,x)_G$.
 \end{proof}

 \begin{cor}{\label{C:connected}}
  Let $X$  and $G$ be as in Theorem \ref{T:fundgpquotient1} (also
  see Remark \ref{R:automatic}). Assume that there is $x\in X$ that
  is fixed by the action. Also, assume that for every $g \in G$,
  the fixed set $X^g$ is path-connected. (It is enough to assume
  that $X^g/C(g)$ is path-connected.) Then, we have an isomorphism
    $$\pi_1(X,x)_G \risom \pi_1(X/G,\bar{x}).$$
 \end{cor}

 \begin{proof}
  Note that the fixed point set $X^g$ contains $x$, for every $g\in
  G$. Hence, we can choose the paths $\{\gamma_{g,i}\}$ to be the
  constant paths (see the construction of $\bar{K}$ just before
  Theorem \ref{T:fixed2} for notation). This way, the elements
  $\bar{\lambda}_{g,i}$ will  be trivial. Therefore, the group
  $\bar{K}$ is trivial.
 \end{proof}

 \begin{ex}{\label{E:BG}}
   Let $G$ be a compact Lie group, and consider the conjugation
   action of $G$ on its classifying space $BG$. (Here, $BG$ stands
   for the classical classifying space, not the stack one.) Using
   the standard model for $BG$, namely, the geometric realization of
   the simplicial space associated to $G$, we see that, for every $g
   \in G$, the fixed set of $g$ is homeomorphic to $B(C(g))$, where
   $C(g)$ is the centralizer of $g$. It follows that the fixed point
   sets are all connected. So, by Corollary \ref{C:connected}, we
   have
    $$\pi_1(BG/G)\cong (\pi_1BG)_G\cong(G/G^0)^{ab}=\pi_0(G)^{ab}.$$
 \end{ex}

 \begin{rem}{\label{R:trick}}
  Theorems \ref{T:fixed} and \ref{T:fixed2} can be generalized to
  the case where the base point $x$ is not necessarily fixed by the
  entire action. For instance, we can take
  $Y=\operatorname{Cone}(G\cdot x)\cup_{G\cdot x}X$ with its induced
  $G$ action. The tip of the cone $\operatorname{Cone}(G\cdot x)$ is
  now fixed by the entire action. Hence,  Theorems \ref{T:fixed} and
  \ref{T:fixed2} apply. Note that  $Y/G\cong [0,1]\cup_{\{0\}} X/G$
  is homotopy equivalent to $X/G$.
 \end{rem}

\section{Relation with Armstrong's work}{\label{S:Armstrong}}

In \cite{Armstrong2} Armstrong gives a formula for the fundamental
group of the coarse quotient $X/G$ of a group $G$ acting faithfully
on a topological space $X$. To do so, he makes three assumptions on
the group action. We will not go over Armstrong's result. Instead,
we translate his conditions  $\mathbf{A}$, $\mathbf{B}$, and
$\mathbf{C}$ into stack language and deduce the stack version of his
result, generalizing the main theorem of \cite{Armstrong2}.

In what follows, $\X$ is a connected locally path-connected
topological stack.

\vspace{0.1in}
 \begin{itemize}

   \item[$\mathbf{A.}$] The moduli map $\pim \: \X \to \Xm$ has the
    homotopy path lifting property. That is, for any point $x \in \X$ and
    any path $\gamma$ initiating from its image $\bar{x} \in \Xm$,
    there is a path $\tilde{\gamma}$ initiating at $x$ such that
    $\pim(\tilde{\gamma})$ is homotopic to $\gamma$ rel. $\bar{x}$.

   \item[$\mathbf{C.}$] Let $N \subseteq \pX$ be as in Lemma
    \ref{L:FPRcovers}. There is a covering stack of $\X$
    corresponding to $N$, and this covering stack is free.
 \end{itemize}

\vspace{0.1in}

 \begin{prop}{\label{P:Amstrong}}
  Assume $\mathbf{A}$ and $\mathbf{C}$ hold. Then, we have a natural
  isomorphism
        $$\pi_1(\X,x)/N \risom \pi_1(\Xm,x).$$
 \end{prop}

 \begin{proof}
  First of all, note that, by $\mathbf{C}$, the prodiscrete topology
  $\mathsf{Free}$ on $\pX$ is the same as the topology generated by
  $N$. So $\pi'_1(\X,x)^{\wedge}_{\mathsf{Free}}=\pX/N$. Let $f\: (\Y,y)
  \to (\X,x)$ be the covering stack corresponding to $N$. By
  $\mathbf{C}$, $\Ym \to \Xm$ is a covering map. Condition
  $\mathbf{A}$ implies that $\pim \: \Y \to \Ym$ also has the
  homotopy path lifting property. Therefore, $\pi_1(\Y,y) \to
  \pi_1(\Ym,y)$ is surjective.

  On the other hand, it follows from the commutative  diagram
         $$\xymatrix@=12pt@M=10pt{
           \Y \ar[r]^f \ar[d]  &  \X \ar[d]  \\
                 \Ym   \ar[r]_{f_{mod}}  &   \Xm  }$$
  that the image of $\pi_1(\Y,y)$ in $\pi_1(\Ym,y)$ is trivial (note
  that $f_{mod}$ is a covering map). This means that $\pi_1(\Ym,y)$ is
  trivial, so $\Xm$ has a universal cover. Therefore, the full
  topology on $\pXm$ is discrete, i.e., $\pi'_1(\X,x)=\pX$. On the
  other hand, by definition, $N$ is open in the full topology of
  $\pX$. The result now follows from
  Proposition \ref{P:fundgpcoarsemod}.$\mathbf{ii}$.
 \end{proof}

In practice, conditions $\mathbf{A}$ and $\mathbf{C}$ may fail for
pathological reasons. What condition $\mathbf{B}$ of Armstrong
requires is the existence of another topological stack $\X'$
together with a map $\X \to \X'$ such that $\X'$ satisfies
$\mathbf{A}$ and $\mathbf{C}$ and that the induced map $\Xm \to
\X'_{mod}$ is so that  every fiber has  trivial topology. Since $\Xm
\to \X'_{mod}$ is a homotopy equivalence, to compute $\pi_1(\X)$ it
is enough to compute $\pi_1(\X')$, and for this
we can apply the above proposition.

A typical example of the above situation is the following. Assume
$G$ is a group acting on a space $X$, and let $\X=[X/G]$. It may
happen that the group $G$ is not nice enough so as to be able to
apply Theorems \ref{T:fundgpquotient1} or \ref{T:fundgpquotient2}.
However, if we can realize $G$ as a dense subgroup of a Lie group
$G'$ and extend the action to $G'$, then chances are that
$\X'=[X/G']$ is a better behaved stack to which the discussion of
the previous paragraph applies.

 \begin{ex}{\label{E:Armstrong}}
  Let $\bbQ$ (viewed as a discrete group) act on $\bbR$ by,
  translation and let $\X=[\bbR/\bbQ]$. Then, $\X'=[\bbR/\bbR]$
  together with the obvious map $\X \to \X'$ has the above property.
 \end{ex}

What happens in this example is that we can replace $\bbQ$ by its
closure in the group $\Iso(\bbR,\bbR)$, thereby turning an awkward
action of an infinite discrete group into a nice action
of a connected Lie group.

In general, when $G$ is a subgroup of the group of isometries of a
locally compact metric space $X$, then the action of the closure
$\bar{G}$ (in the compact-open topology) satisfies the above
property (\cite{Armstrong2}, Corollary 2). That is, the fibers of
the map $q \: X/G \to X/\bar{G}$ have trivial topology (hence $q$ is
a homotopy equivalence). In this situation, Armstrong's
trick comes handy.


\providecommand{\bysame}{\leavevmode\hbox
to3em{\hrulefill}\thinspace}
\providecommand{\MR}{\relax\ifhmode\unskip\space\fi MR }
\providecommand{\MRhref}[2]{%
  \href{http://www.ams.org/mathscinet-getitem?mr=#1}{#2}
} \providecommand{\href}[2]{#2}

\end{document}